\documentclass[12pt,reqno]{amsart}
\usepackage{hyperref}
\usepackage{amsmath,amsfonts,amssymb}
\usepackage[latin1]{inputenc}

\newtheorem{theorem}{Theorem}[section]
\newtheorem{lemma}[theorem]{Lemma}
\newtheorem{prop}[theorem]{Proposition}
\newtheorem{cor}[theorem]{Corollary}
\newtheorem{rmk}[theorem]{Remark}
\newtheorem{defn}[theorem]{Definition}

\newcounter{defn}

\newcommand{\sect}{\vspace{3mm} \setcounter{equation}{0} \setcounter{defn}{0} \section}

\newcommand{\w}[1]{\langle {#1} \rangle}
\newcommand{\pf}{\noindent {\bf Proof. \hspace{2mm}}}
\newcommand{\ef}{ \hfill $ \Box $ \vskip 3mm}
\newcommand{\be}{\begin{equation}}
\newcommand{\ee}{\end{equation}}
\newcommand{\bea}{\begin{eqnarray}}
\newcommand{\eea}{\end{eqnarray}}
\newcommand{\ep}{{\epsilon}}
\newcommand{\f}{\frac}
\newcommand{\na}{\nabla}

\newcommand{\bC}{{\mathbb C}}
\newcommand{\bR}{{\mathbb R}}
\newcommand{\bN}{{\mathbb N}}

\newcommand{\vB}{{\mathcal B}}

\newcommand{\vF}{{\mathcal F}}

\newcommand{\vH}{{\mathcal H}}

\newcommand{\vL}{{\mathcal L}}

\newcommand{\vN}{{\mathcal N}}

\newcommand{\vS}{{\mathcal S}}

\newcommand{\gm}{{\frak m}}

\def\p{\partial}

\def\f{\frac}
\def\Dl{\Delta}
\def\na{\nabla}

\def\Dl{\Delta}

\def\i{\infty}

\begin{document}

\title[The Kramers-Fokker-Planck equation]{Large-time asymptotics of solutions \\
 to the Kramers-Fokker-Planck equation \\
 with  a short-range potential}
\author{Xue Ping  WANG}
\address{Laboratoire de Mathématiques Jean Leray\\
UMR CNRS 6629\\
Université de Nantes \\
44322 Nantes Cedex 3  France \\
E-mail: xue-ping.wang@univ-nantes.fr}

\thanks{Research  supported in part by French ANR Project NOSEVOL BS01019 01 and by Chinese Qian Ren programme at Nanjing University.}

\begin{abstract}

In this work, we use scattering method to study the Kramers-Fokker-Planck equation with a potential whose gradient tends to zero at the
infinity.  For short-range potentials in dimension three, we show that complex eigenvalues do not accumulate at low-energies and establish the low-energy resolvent asymptotics. This combined with high energy pseudospectral estimates  valid in more general situations gives the large-time asymptotics of the solution in weighted $L^2$ spaces.
\end{abstract}

\maketitle

\sect{Introduction}

The Kramers equation (\cite{risc}), also called the Fokker-Planck equation (\cite{dv,hln,hrn}) or the Kramers-Fokker-Planck equation (\cite{hhs,hss}),  is the evolution equation for the distribution functions describing  the Brownian motion of particles in an external field :
\be \label{kfp1}
\frac{\partial W}{\partial t} = \left( -v\cdot \nabla_x  + \nabla_v \cdot(\gamma v  -\f{F(x) }{m})
+ \f{\gamma k T }{m} \Delta_v \right)W,
\ee
where $W=W(t; x,v)$,  $x, v \in \bR^n$, $t \ge 0$ and $F(x) = - m \nabla V(x)$ is the external force.
In this equation,  $x$ and $v$  represent  the position and velocity of particles, $m$ the mass, $k$ the Boltzmann constant, $\gamma$ the friction coefficient and $T$ the temperature of the media.  This equation is a special case of the more general  Fokker-Planck equation (\cite{risc}).
After a change of unknowns and for appropriate values of physical constants, the Kramers-Fokker-Planck  equation (\ref{kfp1})  can be written into the form (\cite{hln,risc})
\be\label{equation} \p_t u(t; x,v)+P u(t; x,v)=0,\ (x,v)\in\bR^n\times\bR^n, n \ge 1,  t >0,
\ee
with some initial data
\be\label{initial} u(0; x,v)=u_0(x,v). \ee
Here $P$ is the Kramers-Fokker-Planck (KFP, in short)  operator:
\be\label{operator} P=-\Dl_v+\f{1}{4}|v|^2-\f{n}{2} + v\cdot\na_x-\na V(x)\cdot\na_v, \ee
where the potential $V(x)$ is supposed to be  a real-valued $C^1$ function.\\

The large-time asymptotics of the solution is motivated by the trend to  equilibrium in statistical physics and is studied by several authors in the case where $|\nabla V(x)| \to + \infty$ as $|x|\to + \infty$.  See  \cite{aggmms,dv,hkn,hln,hrn,hhs,v} and references quoted therein. Note that in the work \cite{hhs} on low-temperature analysis of the tunnelling effect, only the condition $|\nabla V(x)| \ge C>0$ outside some compact set is needed. In these cases,  the spectrum of $P$ is discrete in a neighbourhood of zero and nonzero eigenvalues of $P$ are of strictly positive real parts. If in addition $V(x)>0$ outside some compact set, the square-root, $\gm$, of the Maxwilliam is an eigenfunction of $P$ associated with the eigenvalue zero, where $\gm$ is defined by
\be
\gm(x,v)= \f{1}{(2\pi)^{\f n 4}}e^{-\f 1 2 (\f{v^2}{2} + V(x))}.
\ee
It is proven in these situations that
the solution $u(t)$ to (\ref{equation}) with initial data $u_0$ converges exponentially rapidly to $\gm$  
in the sense that there exists  some $\sigma>0$ such that for any nice initial data  $u_0$, one has in appropriate spaces
\be \label{eq1.6}
u(t)= \w{\gm, u_0} \gm + O(e^{-\sigma t}), \quad t\to +\infty,
\ee
where $V(x)$ is assumed to be normalized by 
\[
\int_{\bR^n} e^{-V(x)} dx =1
\]
 and $\sigma$ can be evaluated in some regimes in terms of the spectral gap between  zero and the real parts of other nonzero eigenvalues.  Since the change of the unknown from (\ref{kfp1}) to (\ref{equation})
is essentially given by $W(t) = \gm u(t)$ with appropriate choice of physical constants, this result shows that the distribution functions governed by (\ref{kfp1}) always tend, up to some multiplicative constant  depending only on the initial data, to the Maxwillian $\gm^2$, as $t \to +\infty$ and justifies the well-known phenomenon of the return to equilibrium in statistical physics. 
Since the KFP operator is neither elliptic nor selfadjoint,  the proof of such result is highly nontrivial and is realized first  by the entropy method in \cite{dv} (see also \cite{v}) and later on by microlocal and spectral methods in \cite{hln,hrn,hhs}. If $V(x)$ is slowly increasing so that $|\nabla V(x)|\to 0$ as $|x|\to\infty$, (for example, $V(x) \sim c \w{x}^\mu$ or $V(x) \sim a \ln |x|$ as $|x|\to \infty$ for some $0<\mu <1$ and $a, c>0$), $\gm$ is still an eigenfunction of $P$ associated with the eigenvalue zero (in the second case it may be an eigenfunction for some values of $a$ and a resonant state for some other values of $a$), but now the essential spectrum of $P$ is equal to $[0, +\infty[$ and there is no spectral gap between the eigenvalue zero and the other part of the spectrum of $P$. A natural question at this connection is whether there still exists some phenomenon of the return to equilibrium in such cases. \\

 The goal of this work is to study  spectral properties of the KFP operator and  large-time asymptotics of the solutions to the KFP  equation  with a potential $V(x)$ such that $|\nabla V(x)| \to 0$ as $|x|\to +\infty$.  Although our final result concerns only short-range potentials in dimension three, some results hold  for
slowly increasing potentials in any dimension. Throughout this work,   we assume that $V$ is $C^1$ on $\bR^n$ and
\be \label{ass}
|\nabla V(x)| \le C\w{x}^{-1-\rho},  \quad x\in \bR^n,
\ee
for some $\rho \ge -1$, where $\w{x} = (1 + |x|^2)^{1/2}$. The potential is said to be slowly increasing if $-1 < \rho \le 0$ and if it is positive near outside some compact, and to be of short-range if $\rho>1$. The KFP operator  $P$ with the maximal domain in $L^2$ is a closed, accretive ( $\Re P \ge 0$) and hypoelliptic operator.  Denote
\be
P = P_0 + W,
\ee
 with $ P_0 = v\cdot\na_x-\Dl_v+\f{1}{4}|v|^2-\f{n}{2}$  and $ W= -\na V(x)\cdot\na_v$.
 If $\rho>-1$, $W$ is a relatively compact perturbation of the free KFP operator $P_0$: $ W (P_0+1)^{-1}$ is a compact operator in $L^2$. One can check that the essential spectrum of $P$ is equal to $\sigma(P_0) = [0, +\infty[$ and the non-zero complex eigenvalues of $P$ have strictly positive real parts and may accumulate towards some points in the essential spectrum. 
\\

The main results of this paper can be summarized as follows.

\begin{theorem}\label{th1.1}
 (a). Assume  $n \ge 1$ and  the condition (\ref{ass}) with $\rho \ge -1$. 
Then there exists $C>0$ such that   $\sigma (P) \cap \{z; |\Im z| > C,  \Re z \le \f 1 C |\Im z|^{\f 1 3}  \} = \emptyset$ and
the resolvent $R(z) =(P-z)^{-1}$ satisfies the estimates
\be \label{eq1.21}
\|R(z)\| \le \f{C}{|z|^{\f 1 3}}, \quad
\ee
and
\be \label{eq1.22}
\|(1-\Delta_v + v^2)^{\f 1 2}R(z)\| \le \f{C}{|z|^{\f 1 6}}, \quad
\ee
for $|\Im z| > C $ and $ \Re z \le \f 1 C |\Im z|^{\f 1 3}$.
\\

(b). Assume  $n=3$ and $\rho >1$.  Then  $P$ has no eigenvalues  in a neighborhood of $0$. Let $S(t) =e^{-tP}$ be the semigroup of contractions generated by $-P$. One has 
\be \label{eq1.23}
\|S(t)\|_{\vL^{2,s} \to \vL^{2,-s}} \le C  t^{- \f 3 2}, \quad t  >0.
\ee
for any $s>\f 3 2$.  Here $\vL^{2,s} = L^{2}(\bR^{2n}_{x,v}; \w{x}^{2s} dx dv)$. \\

(c). Assume  $n=3$ and $\rho >2$. Then for  any $s> \f 3 2$, there exists $\ep>0$ such that  one has the following asymptotics
\be \label{eq1.24}
S(t) =   \frac{1}{ (4\pi t)^{\f 3 2}}\w{\gm, \cdot} \gm  + O(t^{- \f 3 2 -\ep}), \quad t \to + \infty,
\ee
as operators from $ \vL^{2, s} $ to $\vL^{2, -s}$. 
\end{theorem}

It may be interesting to compare (\ref{eq1.24}) with (\ref{eq1.6}). The space distributions of solutions to (\ref{kfp1}) in both cases are governed by the Maxwillian, but for decreasing potentials, the density of distribution decays in time like $t^{-\f 3 2}$.  Recall that it is well-known for Schr\"odinger operator  $H= -\Delta_x + U(x)$ with a real-valued potential $U(x)$ that the space-decay rates of solutions to  $Hu=0$ determine the 
low-energy asymptotics of the resolvent $(H-z)^{-1}$ near the threshold zero, which in turn determine large-time asymptotics of  solutions to the evolution equation (see, for example, \cite{w0}).  From this point of view, the  difference between (\ref{eq1.6}) and (\ref{eq1.24}) may be explained by the lack of decay in $x$-variables in the Maxwillian $\gm^2$ in our case and one may even expect that the solution to the KFP equation behaves like $t^{-\alpha}$ with $\alpha$ depending on $a$ for critical potentials $V(x) \sim a\ln |x|$, $a \in ]\f{n-2}{2}, \f n 2[$.  Notice also that different from (\ref{eq1.6}) there is no normalization for $V(x)$ in (\ref{eq1.24}). Although the KFP operator is a differential operator of the first order in $x$-variables,  the large-time behavior   of solutions looks like those to the heat equation described by  $e^{t \Delta_x}$ as $t\to+\infty$. This  is due to the interplay between the diffusion part and the transport part of the KFP operator $P$ and will become clear from spectral decomposition formula of $e^{-tP_0}$. Under stronger assumption on $\rho$, one can calculate the second term in large-time asymptotics of solutions which is of the order $O(t^{-\f 5 2})$ in dimension three.\\

The method used in this work is  scattering in nature: we regard the full KFP operator $P$ as a perturbation of the free KFP operator $P_0$ without potential. A large part of this work is devoted to a detailed analysis of the free KFP operator. Several basic questions  remain open for $P_0$ such as the high energy estimates for the free resolvent near the positive real axis.
 A key step in the proof of the large-time asymptotics of solutions to the full KFP equation is to show that the complex eigenvalues of $P$ do not accumulate towards the threshold zero. So far as the author  knows,  such a statement is not yet proven for non-selfadjoint Schr\"odinger operators $-\Delta + U(x)$ with  a general complex-valued potential satisfying $|U(x)| = O(|x|^{-\rho}), \rho >2$. (See however \cite{ls,w1} for dissipative potentials ($\Im U(x) \le 0$)). 
Results like (\ref{eq1.23}) and (\ref{eq1.24}) may fail if there is a sequence of complex eigenvalues tending to zero tangentially to the imaginary axis. For the KFP operator $P$ with a short-rang potential in dimension three, we prove that there are no complex eigenvalues in  neighborhood of zero by making use of the method of threshold spectral analysis and the supersymmetric structure of the operator. \\

The organisation of this work is as follows.  
In Section 2, we study in detail the spectral properties of the model operator $P_0$  and establish some dispersive estimate for the semigroup generated by $-P_0$. We also prove  the limiting absorption principles  and the low-energy asymptotics of the resolvent $R_0(z) =(P_0 -z)^{-1}$, as well as some high energy resolvent estimates. these estimates are pseudo-spectral in nature, because the numerical range of the free KFP operator is equal to the right half complex plane.   
The threshold spectral properties of the full KFP operator $P$ with a short-range potential is analyzed in Section 3. We prove that there is no eigenvalue of $P$ in a neighborhood of zero and calculate the lower-energy resolvent asymptotics.
For technical reasons, we only prove these results when the dimension $n$ is equal to three, but we believe that they  remain true when $ n\ge 4$. Finally, in Section 4, we prove a high energy pseudo-spectral estimate in general situation and this combined with the low energy resolvent estimates obtained in dimension three allows to prove  the time-decay and the large-time asymptotics of  solutions. In Appendix A,  we study a family of nonselfadjoint  harmonic oscillators which may be regarded as complex translation in variables of selfadjoint harmonic oscillators. We prove some quantitative estimates with respect to the parameters of translation, establish a spectral decomposition formula and prove some uniform time-decay estimates of the semigroup. These results are used in Section 2 to analyze the free KFP operator.

\sect{The free Kramers-Fokker-Planck operator}

 Denote by $P_0$ the free KFP operator (with $\nabla V=0$):
\be
P_0=v\cdot\na_x-\Dl_v+\f{1}{4}|v|^2-\f{n}{2}.
\ee
$P_0$ is a  non-selfadjoint and hypoelliptic  operator with loss of $\f 1 3$ derivative in $x$ variables. The following result is known (see \cite{hln,n}). In particular, the essential maximal accretivity is discussed in Proposition 2.4 of \cite{n}.

\begin{prop} \label{prop2.1}  One has
\be
\|\Delta_v u \|^2 + \||v|^2u\|^2 + \| |D_x|^{\f 23 } u\|^2 \le C ( \|P_0u \|^2 + \|u\|^2), \quad  u\in \vS(\bR_{x,v}^{2n})
\ee
$P_0$ defined on $\vS(\bR^{2n}_{x,v})$ is essentially maximally accretive, {\it i.e.}, the closure of $P_0$ in $L^2(\bR^{2n}_{x,v})$  with core $\vS(\bR^{2n}_{x,v})$ is of maximal domain $D(P_0) = \{u \in L^2(\bR_{x,v}^{2n}); P_0u \in L^2(\bR_{x,v}^{2n})\}$ and $\Re \w{P_0u, u} \ge 0$ for $u\in D(P_0)$.
\end{prop}

Henceforth we still denote by $P_0$ its closed extension in $L^2$ with maximal domain $D(P_0)= \{u \in L^2(\bR_{x,v}^{2n}); P_0u \in L^2(\bR_{x,v}^{2n})\}$.\\

In terms of Fourier transform in $x$-variables, we have for $u \in D(P_0)$
\bea
 P_0 u (x, v)& = &\vF_{x\rightarrow\xi}^{-1}\hat{P}_0(\xi) \hat{u}(\xi, v), \quad \mbox{ where }\\
\hat{P}_0(\xi) &=&-\Dl_v+\f{1}{4}\sum^n_{j=1}(v_j+2i\xi_j)^2-\f{n}{2}+|\xi|^2 \\
 \hat{u}(\xi, v) & = & (\vF_{x\rightarrow\xi}u)(\xi, v) \triangleq \int_{\bR^n} e^{-i x\cdot\xi}u(x, v) \; dx.
 \eea
Denote
\be
D(\hat{P}_0) =  \{f \in L^2(\bR^{2n}_{\xi, v}); \hat{P}_0(\xi)f   \in L^2(\bR^{2n}_{\xi, v})\}.
\ee
Then $\hat{P}_0  \triangleq\vF_{x\rightarrow\xi} P_0 \vF_{x\rightarrow\xi}^{-1}$ is a direct integral of the  family of complex harmonic operators $\{\hat{P}_0(\xi); \xi \in \bR^n \}$ which is studied in Appendix A.
 \\
 
The following abstract result may be useful to determine the spectrum of operators which are direct integral of a family of nonselfadjoint operators.  See also Theorem 8.3 of \cite{d4} for families of bounded nonselfadjoint operators. 

\begin{theorem} \label{th3.1} Let $H$ be a separable Hilbert space, $X $ a non empty open set of $\bR^n$ and $\vH = \{f : X \to H; \|f\| = (\int_X \|f(x)\|_H^2 dx\}^{\f 1 2} < + \infty\}$. Here $dx$ is the Lebesgue measure of $\bR^n$. Suppose that  both $\{Q(x); x\in X\}$  and the adjoints  $\{Q(x)^*; x\in X\}$  are  strongly continuous families of closed, densely defined operators with constant domains $D, D^* \subset H$, respectively.
Suppose in addition that  for each $z \in \bC \setminus \overline{ \cup_{x\in X} \sigma (Q(x))}$, one has
\be \label{e3.2}
\sup_{x\in X} \| (Q(x) -z)^{-1} \| < +\infty.
\ee
Let $Q$ be the closed, densely defined operator in $\vH$ such that for any $f$ in the domain of $Q$, one has
$f(x) \in D$ about everywhere and
\be
Qf = Q(x)f(x) \quad \mbox{ in } \vH.
\ee
Then one has
\be \label{eq3.3}
\sigma(Q) = \overline{ \cup_{x\in X} \sigma (Q(x))}
\ee
\end{theorem}
\pf If $z \not\in \overline{ \cup_{x\in X} \sigma (Q(x))}$, define $R_z : \vH \to \vH$ by $(R_z f)(x) = (Q(x)-z)^{-1}f(x)$, $f\in \vH$. Then
(\ref{e3.2}) shows that $R_z$ is bounded on $\vH$. One can check that $(Q-z)R_z=1$ on $\vH$ and $R_z (Q-z) =1$ on $D(Q)$.
Thus $z$ is in resolvent set of $Q$. This shows that $\sigma(Q) \subset \overline{ \cup_{x\in X} \sigma (Q(x))}$. \\

Conversely, if $z \in \sigma(Q(x_0))$ for some $x_0 \in X$, then either $\| (Q(x_0) -z)u_n\| \to 0$ or
$\| (Q(x_0)-z)^* v_n\| \to 0$ for some sequences $\{u_n\}$, $\{v_n\}$ in $H$ with $\|u_n\| = \| v_n\| =1$. In fact if one had   both $\| (Q(x_0) -z)u\| \ge c \|u\|$ and $\| (Q(x_0) -z)^*v\| \ge c \|v\|$ for some $c>0$ and for all $u \in D$ and $v \in D^*$, $z$ would belong to the resolvent set of $Q(x_0)$.
Since $x \to Q(x)$ is strongly continuous on $D$, for any $\epsilon >0$, we can find $f \in \vH $ in the form $f = \chi(x) u_{n_0}$ (or
$g = \chi(x)v_{n_0} \in \vH$ in the second case $\| (Q(x_0)-z)^* v_n\| \to 0$)  for some $n_0$ and for some numerical function $\chi$ supported in a sufficiently small neighborhood of $x_0$ such that $\|f\|=1$ and  $\| (Q-z)f\|<\epsilon$ (
or  $\|g\|=1$ and  $\| (Q-z)^*g\|<\epsilon$ ). This shows that $z \in \sigma (Q)$. Consequently,
$\sigma(Q) \supset \overline{ \cup_{x\in X} \sigma (Q(x))}$.
\ef

  Remark that if the condition (\ref{e3.2}) is not satisfied, the equality (\ref{eq3.3}) may fail. See Theorem 8.3 of \cite{d4}.

\begin{prop} \label{prop3.2} Let $P_0$ denote the free KFP operator with the maximal domain.
Then one has: $\sigma(P_0) = [0, + \infty[$.
\end{prop}
\pf We want to apply Theorem \ref{th3.1}.
 One sees that $P_0$ is unitarily equivalent with $\hat{P_0}$ which is a direct integral of a family of operators
 $\{  \hat{P_0}(\xi); \xi \in \bR^n \}$ with constant domain $D$. Lemma \ref{lem2.2} shows that
 \[
  \cup_{\xi\in \bR^n } \sigma(\hat{P_0}(\xi)
  ) = [0, + \infty[.
 \]
 It is clear that $ D( \hat{P_0}(\xi)^*) =D$  and $x \to \hat{P_0}(\xi) $ and  $x \to \hat{P_0}(\xi)^* $ are strongly continuous on $D$. To apply Theorem \ref{th3.1} to show that $\sigma(P_0) = [0, + \infty[$, it remains to check the condition (\ref{e3.2}):  for each $z \not\in [0, +\infty[$, there exists some constant $C_z$ such that
 \be \label{eq2.11a}
 \|(\hat{P_0}(\xi) -z)^{-1}\| \le C_z
 \ee
 uniformly in $\xi \in \bR^n$.  For $\xi$ in a compact, this follows from the fact that since $\hat{P_0}(\xi) $ forms a holomorphic family of type (A) in sense of Kato, the resolvent $(\hat{P_0}(\xi) -z)^{-1}$ is locally bounded in $\xi \in \bR^n$ for each $z\not\in [0, + \infty[$ (\cite{Kat80}).   For $|\xi|$ large ($|\xi|^2 >|\Re z| +1$), using the representation
\[
(\hat{P_0}(\xi) -z)^{-1}  = - \int_0^T e^{-t(\hat{P_0}(\xi) -z)} \; dt -  \int_T^\infty e^{-t(\hat{P_0}(\xi) -z)} \; dt
\]
with $T \ge 3$  fixed, one deduces from Corollary \ref{cor2.4}  that there exists $ C= C(\Re z)$ independent of $\xi$   such that
\be
\| (\hat{P_0}(\xi) -z)^{-1} \| \le C + \f{C}{\xi^2- \Re z}
\ee
for $\xi^2 > |\Re z| +1$.
 This proves (\ref{eq2.11a}) which finishes the proof of Proposition \ref{prop3.2}.
\ef

 From Proposition \ref{prop2.3} in Appendix A on a family of non-selfadjoint harmonic oscillators, one can deduce some time-decay estimates for  $e^{-t P_0}$  in appropriate spaces. Denote
\[
\vL^{2,s}(\bR^{2n})=L^2(\bR^{2n}; \langle x\rangle^{2s}dxdv).
\]
and
\[
\vL^p = L^p(\bR^n_x; L^2(\bR^n_v)), \quad p \ge 1,
\]
equipped with their natural norms.\\

\begin{theorem}\label{th3.3}

(a). One has the following dispersive type estimate:  $\exists C>0$ such that
\be
\|e^{-tP_0}u\|_{\vL^\infty}\leq \f{C}{t^\f{n}{2}}\|u\|_{\vL^1}, \quad t\ge 3,
\ee
for $u \in \vL^1$.

(b). For $s>\f{n}{2}$, one has for some $C_s>0$
\be
\|e^{-tP_0}u\|_{\vL^{2,-s}}\leq \f{C_s}{t^\f{n}{2}}\|u\|_{\vL^{2,s}},
\ee
for   $t\ge 3$  and $u\in L^{2,s}$.
\end{theorem}
\pf For $u \in \vS(\bR^{2n}_{x,v})$, we denote by $\hat u$ the Fourier transform of $u$ in $x$-variables.
Then one has
\be
\| e^{-tP_0} u\|_{\vL^\infty}  \le  \frac{1}{(2\pi)^n} \|  e^{-t \hat{P}_0(\xi) } \hat{u}\|_{\vL^1_\xi}
\ee
Here  $\vL_\xi^1$ denotes the space $L^1(\bR^n_\xi; L^2(\bR^n_v))$. 
Corollary  \ref{cor2.4} gives for any $\xi \in \bR^n$
\be
 \|  e^{-t \hat{P}_0(\xi) } \hat{u}\|_{L^2_v} \le  \f{e^{-\xi^2(t-2 - \f{4}{e^t-1}) }}{(1-e^{-t})^n} \|\hat{u}(\xi, \cdot)\|_{L^2_v}.
\ee
Since  $t-2 - \f{4}{e^t-1} \ge c_0 t>0$  for some $c_0>0$ when $t \ge 3$, one obtains that  
\bea
 \|  e^{-t \hat{P}_0(\xi) } \hat{u}\|_{\vL^1_\xi} &\le  & \int_{\bR^n_\xi}\f{e^{-\xi^2(t-2 - \f{4}{e^t-1}) }}{(1-e^{-t})^n} d\xi \|u\|_{\vL^1} \nonumber \\
&\le & C t^{-\f n 2} \|u\|_{\vL^1}
\eea
for $t \ge 3$ and for any $u \in \vS$. An argument of density proves (a) for any $u \in \vL^1$.
\\

 Theorem \ref{th3.3} (b) is a consequence of Theorem \ref{th3.3}  (a).
\ef

For a symbol $a(x,v; \xi, \eta)$, denote by $a^w(x,v, D_x, D_v)$ the associated Weyl pseudo-differential operator defined by
\bea
\lefteqn{a^w(x,v,D_x,D_v)u(x,v)} \\[2mm]
& =& \f{1}{(2\pi)^{2n}} \int\int e^{i(x-x')\cdot\xi + i(v-v')\cdot\eta} a(\f{x+x'}2, \f{v+v'}2, \xi,\eta) u(x'v') dx'dv'd\xi d\eta \nonumber 
\eea
for $u \in \vS(\bR^{2n}_{x,v})$. 
\\

The following Proposition is a spectral decomposition for the semigroup $e^{-t P_0}$ which follows from Proposition \ref{prop2.3} given in Appendix A and some elementary symbolic calculus for Weyl pseudo-differential operators (see \cite{hor}).

\begin{prop} \label{prop3.3a} For any $t>3$, one has
\be\label{expansion1} 
e^{-t P_0 }=\sum_{l=0}^{\i}e^{- l t + (t-2) \Delta_x }p_l^w(v,D_x,D_v), 
\ee
where the series is norm-convergent as   operators in $L^2(\bR^{2n}_{x,v})$ and 
 $p_l(v, \xi,\eta)$ is given by
 \be
 p_l(v, \xi,\eta) = \int_{\bR^n} e^{-i v'\cdot\eta/2}\left(\sum_{|\alpha|=l} e^{-2|\xi|^2} \psi_\alpha( v+ v' + 2i\xi)\psi_\alpha( v- v' + 2i\xi) \right) dv',  \quad l\in \bN.
 \ee
In particular, 
\be
p_0(v, \xi,\eta) = 2^{\f n2}e^{-v^2 -\eta^2 + 2iv\cdot\xi }.
\ee
\end{prop}

 We just indicate that for $t>3$, one has $t-2 - \f{4}{e^t-1} >0$ and
 \be
e^{ (t-2) \Delta_x }p_l^w(v,D_x,D_v)  = e^{t\Delta_x} \Pi_l^{D_x}
\ee 
 is a bounded operator, where $\Pi_l^\xi$ is the Riesz projection of $\hat{P}_0(\xi)$ associated with the eigenvalue $E_l = l+ |\xi|^2$. See Lemma \ref{lem2.2} in Appendix A. The norm-convergence as  operators in $L^2(\bR^{2n}_{x,v})$ of the right-hand side of (\ref{expansion1}) follows from (\ref{eq2.10}) which gives:
\[
\sum_{l=0}^{\i}\| e^{- l t + (t-2) \Delta_x }p_l^w(v,D_x,D_v)\| \le \f{1}{(1-e^{-t})^n}, \quad t >3. 
\]
The detail of the proof is omitted. As an application of this spectral decomposition (\ref{expansion1}), we can establish the following
 result on large-time approximation of solutions to the free KFP equation.  \\
 
\begin{prop}\label{prop3.4} There exists $C>0$ such that 
\be
\|e^{-tP_0}u -e^{(t-2)\Delta_x} p_0^w u\|_{\vL^\infty}\leq C \f{e^{-t}}{t^\f{n}{2}}\|u -e^{-2 \Delta_x}p_0^w u\|_{\vL^1},
\ee
for $t \ge 3$ and for any $u\in \vL^1$ with $e^{-2 \Delta_x}u \in \vL^1$.
\end{prop}
\pf By a direction calculation, one can check that
\bea
\vF^{-1}_{x\to\xi} \Pi_0^\xi \vF_{x\to\xi} &= &e^{-2 \Delta_x} p_0^w, \\
 \vF^{-1}_{x\to\xi}e^{-t\hat P_0(\xi)} \Pi_0^\xi \vF_{x\to\xi} &=& e^{(t-2)\Delta_x} p_0^w.
\eea
$p_0^w$ is continuous on $\vL^1$.  In fact, let $\tau : u (x, v) \to u(x+2v, v)$. Then one has
\[
p_0^wu(x, v)= \w{\psi_0, \tau u}_{L^2(\bR^n_v)}(x) \psi_0(v),
\]
 where $\psi_0 = \frac{1}{(2\pi)^{\f n 4}}e^{-\f{v^2}{4}}$ is the first eigenfunction of $-\Delta_v + \frac{v^2}{4}$.  It follows that
\[
\|p_0^w u\|_{\vL^1} \le \int_{\bR^{2n}} \psi_0(v)|u(x+ 2v, v)|dxdv  =   \int_{\bR^{2n}} \psi_0(v)|u(y, v)|dydv\le \|u\|_{\vL^1}.
\]
 Denote $u_0 = e^{-2\Delta_x}p_0^wu$ for $u \in \vL^1$ with supp$_\xi \hat u$ compact. The proof of Lemma \ref{lem2.3} shows that if $n=1$, one has
\bea
\lefteqn{\sum_{k=1}^{\i}e^{-t(k+\xi^2) }\| \Pi_k^\xi\| } \nonumber \\
&= &  \sum_{k=1}^{\i}e^{-t(k+\xi^2) + 2\xi^2}\sum_{j=0}^{k}\f{C_k^j}{j!}(4\xi^2)^j  \\
&=& \f{e^{- \xi^2(t-2) -t}}{1-e^{-t}} + \f{e^{-\xi^2(t - 2)}}{1-e^{-t}}\sum_{j=1}^\i\f{1}{j!}(\f{4\xi^2}{e^t-1})^j \nonumber\\
& \le & \f{e^{- \xi^2(t-2) -t}}{1-e^{-t}}  \left( 1 + \f{ 4 \xi^2 e^{\f{4\xi^2}{e^t-1} }}{1 -e^{-t}} \right), \quad t >0.  \nonumber 
\eea
When $n\ge 1$, we deduce from the explicit formula for the Riesz projections $\Pi_k^\xi$ and the above one-dimensional bound that there exist some constants $C$, $a >0$ such that
\be \label{eq2.26a}
\sum_{k=1}^{\i}e^{-t(k+\xi^2) }\| \Pi_k^\xi\|  \le C (1+ \xi^2) e^{- a t \xi^2 -t}, \quad t >3.
\ee
Making use of the above estimate and following the proof of Theorem \ref{th3.3} with $u$ replaced by $u-u_0$, one obtains
\bea
\| e^{-tP_0} (u-u_0)\|_{\vL^\infty} & \le & C\sum_{k=1}^{\infty}\| e^{-t(k+\xi^2)}\Pi_k^\xi ( \hat u - \hat u_0)\|_{\vL^1_\xi} \nonumber  \\
& \le & C '  \frac{e^{-t}}{t^{\f n2}} \| u - u_0\|_{\vL^1}, \quad t>3.
\eea
Since $e^{-tP_0}u_0 = e^{(t-2)\Delta_x} p_0^w u$, an argument of density proves the desired result.
\ef

The time-decay of $e^{-tP_0}$ is governed by the first eigenvalue of the harmonic oscillator in $v$-variables and propagation of energy due to the transport term $v\cdot\nabla_x$. Although this term is of the first order in $\xi$, Theorem \ref{th3.3} shows that solutions to the free KFP equation
decay like those to the heat equation in space variables. A natural question is to see if the results  of  Theorem \ref{th3.3} are still true for
the full KFP operator with a potential $V(x)$ such that $|\na V(x)|$ tends to zero sufficiently rapidly. 
In this work, we prove a result similar to  Theorem \ref{th3.3} (b) through resolvent estimates for the full KFP operator $P$. \\

In order to study the resolvent of $P$, we establish here some limiting absorption principles for the resolvent of $P_0$ and its low-energy asymptotics. Different from the limiting absorption principle for selfadjopint operators, the problem  we want to study here is pseudospectral in nature, because $\bR_+$ is located in the interior of the numerical range of $P_0$. Set $R_0(z) = (P_0 -z)^{-1}$, $\hat R_0(z) = ( \hat P_0-z)^{-1}$ and $\hat R_0(z, \xi) = ( \hat P_0(\xi)-z)^{-1}$ for $z \not \in \bR_+$. Then
$R_0(z) = \vF_{x\to\xi}^{-1} \hat R_0(z) \vF_{x\to\xi}$.  Note that $\hat R_0(z)$ is multiplication in $\xi$-variables by $\hat R_0(z,\xi)$.

\begin{prop}\label{prop3.5}  Let $l\in \bN$ and $ l < a <l +1$ be fixed. Take  $\chi \ge 0$ and $\chi \in C_0^\infty(\bR^n_\xi)$ with supp $\chi \subset \{\xi, |\xi|\le a + 4\}$, $\chi(\xi) =1$ when $|\xi| \le a +3$ and $ 0 \le \chi(\xi) \le 1$.  Then one has
\be \label{R0}
\hat R_0(z, \xi) = \sum_{k=0}^l \chi(\xi) \frac{\Pi_k^\xi}{\xi^2+k -z} + r_l(z,\xi),
\ee
for any $\xi \in \bR^n$ and $z \in \bC$ with $\Re z <a$ and $\Im z \neq 0$. Here $r_l(z, \xi)$ is holomorphic in $z$ with $\Re z < a $ verifying the estimate
\be \label{reste}
\sup_{\Re z <    a, \xi \in \bR^n} \|r_l(z, \xi)\|_{\vL(L^2(\bR^n_v))} <\i.
\ee
\end{prop}
\pf  Let $\chi_1= 1-\chi$. For $\Re z<0$, one has
\bea \label{decom}
\hat R_0(z, \xi)  &=& \int_0^\i  e^{-t(\hat P_0(\xi)-z)} \; dt \nonumber \\
& = & \int_0^\i  \chi_1(\xi) e^{-t(\hat P_0(\xi)-z)} \; dt + \int_0^\i  \chi(\xi) e^{-t(\hat P_0(\xi)-z)} \; dt 
\nonumber\\
& \triangleq  & I_1(z,\xi) + I_2(z,\xi).
\eea
Since $\Re P_0(\xi) \ge 0$, it is clear that for each fixed $T$,  $ \int_0^T  \chi_1(\xi) e^{-t(\hat P_0(\xi)-z)} \; dt$ is uniformly bounded in $\xi$ and $z$ with $\Re z \le a$:
\[
\|\int_0^T  \chi_1(\xi) e^{-t(\hat P_0(\xi)-z)} \; dt \|_{\vB(L^2(\bR_v^n))} \le Te^{aT},
\]
for all $\xi \in \bR^n$. Corollary \ref{cor2.4}  with $T=3$ shows that
\be \label{eq3.17}
\|e^{-t (\hat{P}_0(\xi)-z)}\|_{\vB(L^2(\bR_v^n))}\leq C e^{- t(\xi^2 - 2 - e^{-1} - \Re z)}
\ee
for $t\ge 3$. Since $\chi_1$ is supported in $\xi^2 \ge  a +3$, $ I_1(z,\xi)$ is holomorphic in $z$ with $\Re z <a$ and verifies the estimate (\ref{reste}).   To study $I_2(z,\xi)$, we decompose $e^{-t \hat P_0(\xi)}$ as
\[
e^{-t \hat P_0(\xi)} =  J_1(t,\xi) + J_2(t,\xi),
\]
where $J_j(t, \xi) =  e^{-t \hat P_0(\xi)}S_j^\xi$ with $ S_1^\xi =\sum_{k=0}^l \Pi_k^\xi$ and $S_2^\xi = 1 - S_1^\xi$.
For $\Re z <0$, the contribution of $J_1(t, \xi)$ to $\hat{R}_0(z, \xi)$ is
\[
\int_0^\i e^{tz} J_1(t, \xi) dt = \sum_{k=0}^l \frac{\Pi_k^\xi}{\xi^2 + k -z}.
\]
  By (\ref{norm}), one has for $t \ge T >0$
\be
\|J_2(t,\xi)\|_{\vB(L^2(\bR^n_v))} \le  \sum_{k=l+1}^{\i}e^{-t(k+\xi^2)+2\xi^2}\sum_{j=0}^{k}\f{C_k^j}{j!}(4\xi^2)^j \triangleq J_{21}(t,\xi) + J_{22}(t,\xi)
\ee
where
\begin{eqnarray*}
J_{21}(t,\xi)&=& e^{-\xi^2(t-2)}\sum_{j=0}^{l+1}\f{(4\xi^2)^j}{j!}\sum_{k=l+1}^\i C_k^je^{-tk} \\
J_{22}(t,\xi)&=& e^{-\xi^2(t-2)}\sum_{j=l+2}^{\i}\f{(4\xi^2)^j}{j!}\sum_{k=j}^\i C_k^je^{-tk}.
\end{eqnarray*}
 $J_{21}(t,\xi)$ and $J_{22}(t,\xi)$ can be evaluated as in the proof of
Proposition \ref{prop2.3} (see also the proof of (\ref{eq2.26a}) in the case $l=0$) and we omit the details here. One has for some $C, a>0$ 
\bea
J_{21}(t,\xi) & \le & e^{-\xi^2(t-2) -(l+1) t  }\sum_{j=0}^{l+1}\f{(4\xi^2)^j l!}{j!(1-e^{-t})^{l+1}},  \label{eq3.19}\\
J_{22}(t,\xi) & \le & C e^{-a \xi^2 t -(l+1) t  } \label{eq3.20}
\eea
for $t \ge T$. Since $|\xi|$ is bounded on the support of $\chi$, this implies that there exists some constant $C$  such that
\be
\|J_2(t, \xi)\|_{\vB(L^2(\bR_v^n))} \le C e^{ -(l+1)t}
\ee
uniformly in $\xi \in supp \chi$ and $t\ge T$.  We obtain a  decomposition for $\hat R_0(z, \xi) $ when $\Re z <0$:
\be \label{R00}
\hat R_0(z, \xi) = \sum_{k=0}^l \chi(\xi)  \frac{\Pi_k^\xi}{\xi^2+k -z} + r_l(z,\xi),
\ee
where
\[
r_l(z, \xi) = I_1(z, \xi) + \int_0^\i e^{tz}J_2(t, \xi) dt.
\]
By the estimates (\ref{eq3.17}) and (\ref{eq3.20}), $r_l(z, \xi)$ is holomorphic in $z$ with $\Re z <a$ and verifies the estimate (\ref{reste}). Since the both sides of (\ref{R00}) are holomorphic in $z\in \bC\setminus \bR_+$ with $\Re z <a$, this representation formula remains valid for $z$ in this region.
\ef

For $r,s \in \bR$, introduce the weighted Sobolev space
\[
\vH^{r,s} = \{u \in \vS'(\bR^{2n}); (1 + |D_v|^2 + |v|^2 + |D_x|^{\f 2 3})^{\f r 2 } \w{x}^s u \in L^2\}.
\]
Denote $\vB(r,s; r',s')$ the space of continuous linear operators from $\vH^{r,s}$ to $\vH^{r',s'}$. The hypoellipticity of $P_0$ (Proposition \ref{prop2.1}) shows that $(P_0 +1)^{-1} \in \vB(0,0; 2, 0)$. A commutator argument shows that
$(P_0 +1)^{-1} \in \vB(0,s; 2, s)$ for any $s\in \bR$. \\

\begin{cor}\label{cor3.6} Set $R_0(z) = (P_0 -z)^{-1}$, $z \not \in \bR_+$. \\

(a).  Assume $ n\ge 1$. Let $I$ be a compact interval of $\bR$ which does not contain any non negative integer. Then for any $s>\f 1 2$, one has
\be \label{LAP1}
\sup_{\lambda \in I; \epsilon \in ]0, 1]}\|R_0(\lambda \pm i\epsilon)\|_{\vB(0,s; 2,-s)} <\infty
\ee
The boundary values of the resolvent $R_0(\lambda \pm i0) = \lim_{\ep \to 0_+} R_0(\lambda \pm i\epsilon)$ exists
in $\vB(0,s; 2,-s)$ for $\lambda \in I$ and is continuous in $\lambda$.

(b). Assume $n\ge 3$. Let $I$ be a compact interval containing some non negative integer. Then for any $s>1$, one has
\be \label{LAP2}
\sup_{\lambda \in I; \epsilon \in ]0, 1]}\|R_0(\lambda \pm i\epsilon)\|_{\vB(0,s; 2,-s)} <\infty
\ee
for any $k \in \vN$, the limite $R_0(k\pm i0) = \lim_{z \to k, z \not\in \bR_+} R_0(z) $ exists
in $\vB(0,s; 2,-s)$ for any  $s>1$. One has $R_0(0+  i0)  =R_0(0 - i0)$ and  $R_0(k + i0)  - R_0(k - i0)  \in  \vB(0,s; 2,-s)$ for any $s>\frac 1 2$ if $k\ge 1$.
\end{cor}
\pf Proposition \ref{prop3.5} shows that
\be \label{RR0}
 R_0(z) = \sum_{k=0}^l \chi(D_x) \Pi_k^{D_x}(-\Delta_x +k -z)^{-1} + r_l(z),
\ee
for  $z \in \bC$ with $\Re z <a$ and $\Im z \neq 0$ and
 that  $ r_l(z) $ is  bounded on $L^2$ and  holomorphic in $z$ with $\Re z <a$.
 $ \chi(D_x) \Pi_k^{D_x}$, $k=0, 1, \cdots,$  are  Weyl pseudodifferential operators with nice symbols $b_k$ independent of $x$:
\be
\chi(D_x) \Pi_k^{D_x} =  b_k^w(v, D_x, D_v)
\ee
with $b_k(v, \xi,\eta)$ given by
 \be
 b_k(v, \xi,\eta) = \int_{\bR^n} e^{-i v'\cdot\eta/2}\left(\sum_{|\alpha|=k} \chi(\xi) \psi_\alpha( v+ v' + 2i\xi)\psi_\alpha( v- v' + 2i\xi) \right) dv'.
 \ee
 In particular, 
\be
b_0(v, \xi,\eta) = 2^{\f n2}\chi(\xi)e^{-v^2 -\eta^2 + 2iv\cdot\xi + 2\xi^2}.
\ee
These Weyl pseudodifferential operators belong to $\vB(r,s; r',s)$ for any $r, r', s \in \bR$  (\cite{hor}).\\

Since  for any compact interval $I' \subset \bR$, one has
\be
\sup_{\lambda \in I', \epsilon \in ]0,1]}\|\w{x}^{-s} (-\Delta_x - (\lambda \pm i \epsilon))^{-1}\w{x}^{-s}\|_{\vB(L^2(\bR^n_x))} <\infty
\ee
for any $s>1/2$ if $I'$ does not contain $0$ and for any $s>1$ and $n\ge 3$ if $I'$ contains $0$, it follows from
(\ref{RR0}) that for $I \subset ]-\infty, a[$
\be \label{eq3.25}
\sup_{\lambda \in I; \epsilon \in ]0, 1]}\|R_0(\lambda \pm i\epsilon)\|_{\vB(0,s; 0,-s)} <\infty
\ee
for $s>\f 1 2$ if $I \cap \bN =\emptyset$ or $s>1$ and $n\ge 3$ if $I \cap \bN \neq \emptyset$.
Estimates (\ref{LAP1}) and (\ref{LAP2}) follow from (\ref{eq3.25}) and the resolvent equation
\[
R_0(z) = R_0(-1) + (1+z) R_0(-1)R_0(z)
\]
by noticing that $R_0(-1)\in \vB(0, s; 2, s)$ for any $s \in \bR$.  The other assertions  of Corollary \ref{cor3.6} can be proven by making use of the properties  of $(-\Delta_x -(\lambda \pm i0))^{-1}$.
\ef

 The  formula (\ref{RR0}) can also be used to study the threshold asymptotics of the resolvent $R_0(z)$ as
 $z   \to k$, $\Im z\neq 0$, $k \in \bN$.  To simplify calculations, we only consider  the  threshold zero in the case $n=3$.

\begin{prop}\label{prop3.7} Let $n =3$. One has the following low-energy resolvent asymptotics for $R_0(z)$:
for $s, s'> \f 1 2$ and $s+ s' >2$, 
there exists $\ep>0$ such that 
\be \label{eq3.26}
R_0(z) = G_0 + O(|z|^\epsilon), \mbox{ as }  z\to 0, z\not\in\bR_+,
\ee
as operators in $\vB(-1,s; 1, -s')$.  More generally, for any integer $N \ge 1$ and $s >N+\f 1 2$,  there exists $\epsilon>0$
\be \label{eq3.27}
R_0(z) =\sum_{j=0}^N z^{\f j 2}G_j + O(|z|^{\f N 2 + \epsilon}), \mbox{ as }  z\to 0, z\not\in\bR_+,
\ee
as operators in $\vB(-1,s; 1, -s)$. Here the branch of $z^{\f 1 2}$ is chosen such that its imaginary part is positive when $z\not \in \bR_+$ and $G_j \in \vB(-1,s; 1, -s)$ for $s>j+\frac 1 2$, $j \ge 1$.  In particular,
\be \label{G0}
G_0 = F_0 + F_1,
\ee
where $F_0$ is the operator with integral kernel
\be \label{F0}
F_0(x,v; x', v') = \frac{\psi_0(v)\psi_0(v')}{4\pi |x-x'|}
\ee
 and $F_1\in \vB(-1,s; 1,-s')$ for any $s, s' \ge 0$ and $s+s' >\f 3 2$. $G_1 : \vH^{-1,s} \to \vH^{1,-s}$, $s >\f 3 2$,  is an operator of rank one with integral kernel given by
 \begin{equation}
 K_1(x,x';v,v') = \f{i}{4\pi} \psi_0(v)\psi_0(v'). \label{eq2.41}
 \end{equation} 
Here $\psi_0 = (2\pi)^{-\f 3 4} e^{-\f{v^2}4}$ is the first eigenfunction of the harmonic oscillator $-\Delta_v + \f { v^2} 4$. 
\end{prop}
\pf Note that by a complex interpolation, the results on the resolvent $R_0(z)$ in  Corollary \ref{cor3.6} also hold in $\vB(-1, s; \; 1, -s)$. \\ 

For $z \not \in \bR_+$, (\ref{RR0}) with $l=0$  shows that
\be \label{eq3.31}
 R_0(z) =  b_0^w(v,D_x,D_v)(-\Delta_x  -z)^{-1} + r_0(z),
\ee
with $ r_0(z) \in \vB(-1,0; 1,0)$   holomorphic in $z$  when  $\Re z < a$ for some $a \in ]0, 1[$. Here the cut-off $\chi(\xi)$ is chosen such that  $\chi\in C_0^\infty$ and $\chi(\xi) = 1$ in a neighbourhood of $\{ |\xi|^2 \le a\}$. Therefore $r_0(z)$ admits a convergent expansion in powers of $z$ for $z$ near $0$
\[
r_0(z)   = r_0(0) +  z r_0'(0) + \cdots
\]
in $\vB(-1,0; 1,0)$. It is sufficient to analyze the lower-energy expansion of $ b_0^w(v,D_x,D_v)(-\Delta_x  -z)^{-1}$.
\\

The integral kernel of $  b_0^w(v,D_x,D_v)(-\Delta_x  -z)^{-1}$, $z \not\in \bR_+$,  is given by
\be \label{kernel}
K(x,x';v,v';z) = \int_{\bR^3} e^{i\sqrt{z}|y-(x-x')|}\f{1}{4\pi |y -(x-x')|} \Phi(v,v',y) \; dy
\ee
with
\[
\Phi(v,v',y) = (2\pi)^{-\f 9 2}e^{-\f 1 4 (v^2 + v'^2)} \int_{\bR^3} e^{i(y-v-v')\cdot \xi + 2\xi^2}\chi(\xi) \; d\xi.
\]
Since $\chi\in C_0^\infty$, one has the following asymptotic expansion for $K(x,x';v,v';z)$ : for  any $\ep \in [0,1]$ and $N\ge 0$
\be \label{eq3.33}
|K(x,x';v,v';z) - \sum_{j=0}^N z^{\f j 2} K_j(x, x', v, v') | \le C_{N, \ep} |z|^{\f{N+\ep}{2}}|x-x'|^{N-1+\ep}e^{-\f 1 4 (v^2 + v'^2)}
\ee
where
\bea
\label{Kj}
K_j(x, x';v,v')& = &  \frac{i^j}{4\pi} \int_{\bR^3}|y-(x-x')|^{j-1} \Phi(v,v',y) dy.
\eea
Remark  that for $N\ge 1$, $s', s > N + \f 1 2 $ and $0 <\ep < \min\{s,s'\}- N- \f 1 2$, 
\[
\w{x}^{-s}\w{x'}^{-s'}  |x-x'|^{N-1+\ep}e^{-\f 1 4 (v^2 + v'^2)} \in L^2(\bR^{12})
\]
and the same is true if $N=0$ and $s, s'> \f 1 2 $ with $s+ s'>2$. We obtain the asymptotic expansion for
$ b_0^w(v,D_x,D_v)(-\Delta_x  -z)^{-1}$ in powers of $z^{\f 1 2}$ for $z$ near $0$ and $z\not\in \bR_+$.
\be \label{eq3.34}
b_0^w(v,D_x,D_v)(-\Delta_x  -z)^{-1} =\sum_{j=0}^N z^{\f j 2}K_j + O(|z|^{\f N 2 + \epsilon}), \mbox{ as }  
\ee
as operators in $\vB(0,s'; 0, -s)$, $s', s > N + \f 1 2$ (and $s+s'>2$ if $N=0$). By the hypoelltpticity of $P_0$, this expansion still holds in  $\vB(-1,s'; 1, -s)$. This proves (\ref{eq3.27}) with $G_k$ given by
\be
G_{2j} = K_{2j} + \f{r_0^{(j)}(0)}{j!}, \quad G_{2j+1} = K_{2j+1}, \quad j \ge 0.
\ee

To show (\ref{G0}) and (\ref{eq2.41}), note that since $\chi(0)=1$, one has
\[
 \int_{\bR^3} \Phi(v,v',y) \; dy = \chi(0)\psi_0(v)\psi_0(v') = \psi_0(v)\psi_0(v').
\]
One can then calculate that
\bea
K_0(x, x',v,v') & = & \frac{1}{4\pi} \int_{\bR^3}\Phi(v,v', y)\frac{ 1 }{|y-(x-x')|}  dy  \nonumber \\
& = & \frac{1}{4\pi |x-x'|} \psi_0(v)\psi_0(v') \label{eq3.38}   \\
& & + \frac{1}{4\pi} \int_{\bR^3}\Phi(v,v', y)(\frac{ 1 }{|y-(x-x')|} -\frac{ 1 }{|x-x'|})  dy  \nonumber
\eea
and
\bea
K_1(x, x',v,v')&= &  \frac{i}{4\pi} \int_{\bR^3} \Phi(v, v',y) dy = \frac{i}{4\pi}  \psi_0(v)\psi_0(v').
\eea
This shows (\ref{eq2.41}) and that  $G_0 = F_0 + F_1$  with $F_1 = K_{0,1} + r_0(0)$, $K_{0,1}$ being the operator with the integral kernel
\[
K_{0,1}(x,x', v, v') = \frac{1}{4\pi} \int_{\bR^3}\Phi(v,v', y)(\frac{ 1 }{|y-(x-x')|} -\frac{ 1 }{|x-x'|})  dy,
\]
which is a smooth function and 
\[
K_{0,1}(x,x', v, v') = O(\psi_0(v)\psi_0(v')|x-x'|^{-2})
\]
for $|x-x'|$ large. Therefore $K_{0,1}$ is bounded in $B(-1,s; 1,-s')$ for any $s, s'\ge 0$ and $s+s'>\f 3 2$.
This shows that $F_1 = K_{0,1} + r_0(0)$ has the same continuity property, which proves the decomposition (\ref{G0}) for $G_0$.
\ef

 The following high-energy pseudo-spectral estimate is used in the proof of Theorem \ref{th1.1} (a).

\begin{prop}\label{prop3.8} Let $n \ge 1$.  Then for every $\delta>0$, there exists $M>0$ such that
\be \label{eq3.40}
\|R_0(z)\| \le \f{M}{|z|^{\f 1 3}}, \quad
\ee
and
\be \label{eq3.40b}
\|(1-\Delta_v + v^2)^{\f 1 2}R_0(z)\| \le \f{M}{|z|^{\f 1 6}}, \quad
\ee
for $|\Im z| > \delta $ and $ \Re z \le \f 1 M |\Im z|^{\f 13}$.
\end{prop} 
\pf   We first prove that for some constant $C>0$
\be \label{eq3.41}
\|R_0(z)\| \le \f{C}{|z|^{\f 1 3}}, \quad
\ee
for $z= -\f n 2 + i \mu$ with $\mu \in \bR$. 
It suffices to show that
\be \label{eq3.42}
\|\hat R_0(-\f n 2 + i \mu,\xi)\|_{\vB(L^2(\bR^n_v))} \le \f{C}{|z|^{\f 1 3}}, \quad
\ee
uniformly in $\xi\in \bR^n$.  Notice that 
$\hat{P}_0(0)$ is selfadjoint and that one has 
\[
\|\hat R_0(z, 0) \| \le \f 1{|\Im z|} \mbox{ and } \|v\cdot \xi \hat R_0(z, 0) \| \le \f {C'|\xi|}{\sqrt{|\Im z|}}, 
\]
for $\Im z \neq 0$.   Making use of the resolvent equation
\[
\hat R_0(z,\xi) =  \hat R_0(z, 0) - \hat R_0(z,\xi) i v\cdot \xi \hat R_0(z,0),
\]
 one obtains
\be \label{eq3.44}
\|\hat R_0(z,\xi)\|_{\vB(L^2_v)} \le \f{C}{|\Im  z|}, \quad
\ee
if $|\xi| \le c_1 \sqrt{|z|}$ for some $c_1>0$. For $  c |z|^{\f 1 2} \le |\xi| $ with $c>0$ small enough, since we are concerned with estimates for $|z|$ large, we can use  a rotation and a rescaling to reduce to a semiclassical problem. Set $A(h) = -h^2 \Delta_v + \f{v^2} 4 + i v_1$. Then
\[
\|\hat R_0(z,\xi)\| = |\xi|^{-2} \|(A(h) -z')^{-1}\|
\]
where $ h = |\xi|^{-2}$ and $z' = |\xi|^{-2}( \f n 2 + z)$. According to Theorem 1.4 of \cite{dsz}, one has
\be \label{eq3.45}
\|(A(h) -z')^{-1}\| \le C h^{-\f 2 3},
\ee
if $0<h \le h_0$, $ |z'| \le C$ and $ \Re z' \le \f { |\Im z'|^2}4$. In particular, for $z = -\f n 2 + i \mu$ with $\mu$ real,  one has
\be \label{eq3.46}
\|\hat R_0(z,\xi)\|  \le C' h^{\f 1 3} \le  C'' |z|^{-\f 1 3},
\ee
 for $  c |z|^{\f 1 2} \le |\xi| \le  c^{-1} |z|^{\f 3 2}$. This proves (\ref{eq3.41}). Now to prove (\ref{eq3.40}),
 set $z = \lambda + i \mu$ with $\lambda, \mu \in \bR$ and  write
 \[
 R_0(\lambda + i\mu) =   R_0(-\f n 2  + i\mu)  - (\lambda + \f n 2) R_0(-\f n 2  + i\mu) R_0(\lambda  + i\mu).
 \]
 According to (\ref{eq3.41}), 
 \[\|(\lambda + \f n 2) R_0(-\f n 2  + i\mu)\| \le  \f{C|\lambda + \f n 2|}{|z|^{\f 1 3}}
 \le \f 12\]
 if $|\lambda| \le \f 1 M |\mu|^{\f 1 3}$ and $|\mu| \ge M$ for some $ M >1$ large enough. (\ref{eq3.40}) follows from
 (\ref{eq3.41}) and the equation $  R_0(\lambda + i\mu) =  (1 + (\lambda + \f n 2) R_0(-\f n 2  + i\mu))^{-1} 
 R_0(-\f n 2 + i\mu)$ when $|\lambda|\f 1 M |\Im z|^{\f 13}$  with $M>0$ sufficiently large. The estimate (\ref{eq3.40})  for  $\Re z= \lambda < - \f 1 M |\Im z|^{\f 13}$ follows from the accretivity of $P_0$\\
 
 To show (\ref{eq3.40b}), notice that for $z = \lambda + i\mu$ with $\lambda, \mu \in \bR$, one has the identity
\[
\|\nabla_v u \|^2 + \f 1 4 \||v| u\|^2 = (\lambda + \f n 2) \|u\|^2  + \Re \w{(P_0-z) u, u}
\]
for $u \in D$.  One obtains from (\ref{eq3.40})  that
\begin{eqnarray*}
\lefteqn{\|\nabla_v R_0(z) w \|^2 + \f 1 4  \||v| R_0(z) w\|^2 }\\
& \le &  |\lambda + \f n 2|\| R_0(z)w\|^2  + \|u\| \| R_0(z)w\| \\
&\le &  C ( \f{|\lambda |}{ |z|^{\f 2 3}} + \f 1 { |z|^{\f 1 3}}) \|w\|^2, \quad w \in L^2, 
\end{eqnarray*}
  for $|\Im z| > M $ and $ \Re z \le \f 1 M |\Im z|^{\f 13}$.  (\ref{eq3.40b}) is proven. 
\ef

Remark that when $n\ge 3$,  Corollary \ref{cor3.6} shows the existence of the boundary values of the free resolvent $R_0(\lambda \pm i0)$ for any $\lambda>0$.  But so far it is still unclear what kind of upper bound one can expect for $R_0(\lambda \pm i0)$  as $\lambda \to +  \infty$. From Propositions \ref{prop3.7} and \ref{prop3.8}, one can deduce the following

\begin{cor} \label{cor3.9} Let $n=3$. Let $S_0(t)$, $t\ge 0$,  denote the semigroup generated by $-P_0$. Then for any integer $N\ge 1$ and $s> N + \f 1 2$, the following asymptotic expansion holds for some $\ep>0$
\be \label{eq3.56}
e^{-tP_0} = \sum_{k \in \bN,  2k+1 \le N} t^{-\f{2k+ 3}{ 2}} \beta_k G_{2k+1} + O(t^{-\f{N+2}{2}-\ep}), \quad t\to +\infty,
\ee
in $\vB(0,s, 0_s)$. Here $\beta_k$ is some non zero constant. In particular, the leading term $\beta_0G_1$ is a rank-one operator given by
\be
\beta_0G_1 = \f{1}{8 \pi^{\f 3 2} }\w{\gm_0, \; \cdot} \gm_0 : \vL^{2,s} \to \vL^{2, -s}
\ee
for any $s >\f 3 2$. Here $\gm_0(x,v) = 1\otimes\psi_0(v)$.
\end{cor}

The proof of Corollary \ref{cor3.9} uses a representation of the free semigroup $e^{-t P_0}$ as contour integral of the resolvent $R_0(z)$ in the right half complex plane. See Section 4 for more details in the case $V \neq 0$ where we shall prove a similar result for the full KFP operator $P$ (see (\ref{eq1.24})).  In the final step of the proof of (\ref{eq1.24}), we shall apply Proposition \ref{prop3.11} below to compute the leading term. As a preparation for the proof of this proposition, we establish some formulae on the evolution of observables which may be of interest in themselves.\\

\begin{lemma} \label{lem3.10} Let $n \ge 1$. For $t\ge 0$ and $0 \le s \le t$, one has the following equalities as operators from $\vS(\bR^{2n}_{x,v})$ to $L^2(\bR_{x,v}^{2n})$:
\bea
e^{-(t-s)P_0}v_j e^{-sP_0} &=& e^{-t P_0} (v_j \cosh s - 2  \partial_{v_j} \sinh s + 2 (\cosh s -1) \partial_{x_j}) \label{v1} \\
e^{-(t-s)P_0}\partial_{v_j }e^{-sP_0} &=& -\f 1 2 e^{-t P_0} (  (v_j \sinh s - 2  \partial_{v_j} \cosh s   + 2 \partial_{x_j} \sinh s
 ))  \label{v2} \\
e^{-(t-s)P_0}x_j e^{-sP_0} &=& e^{-t P_0} (x_j+ v_j \sinh s - 2 (\cosh s -1)\partial_{v_j} \nonumber \\
& &  + 2 (\sinh s -s) \partial_{x_j})  \label{x1} 
\eea
\end{lemma}
\pf For fixed $t>0$, set $f(s) = e^{-(t-s)P_0}v_j e^{-sP_0}$, $0 \le s \le t$. Proposition \ref{prop2.1} shows that for $u\in\vS$, $A^k e^{-tP_0}u \in L^2$ for any $k\in \bN$, where $A$ may be one of the operators $v_j, \partial_{v_j}, \partial_{x_j}$, $j=1, \cdots, n$.  As  operators from $\vS$ to $L^2$, one has:
\bea 
f'(s) &=& e^{-(t-s)P_0}[P_0, v_j] e^{-sP_0}  = -2 e^{-(t-s)P_0}\partial_{v_j }e^{-sP_0} \label{v11} \\
f''(s)& = & -2 e^{-(t-s)P_0}\left[\f{v^2} 4 + v\cdot\partial_{x}, \; \partial_{v_j}\right] e^{-sP_0}  =  f(s) + 2 \partial_{x_j} e^{-tP_0} \label{v12}
\eea
This shows that $f(s) = C_1 e^s + C_2 e^{-s} - 2  \partial_{x_j} e^{-tP_0} $.  $C_1$, $C_2$ can be determined by the initial data $f(0)= e^{-tP_0}v_j$ and $f'(0) = -2 e^{-tP_0} \partial_{v_j }$:
\[
C_1 = e^{-tP_0} ( \f 1 2 v_j - \partial_{v_j} +  \partial_{x_j}) , \quad C_2 =  e^{-tP_0} ( \f 1 2 v_j + \partial_{v_j}  + \partial_{x_j}) .
\]
This proves (\ref{v1}).  (\ref{v2}) follows from (\ref{v1}) and the equality
\[
 e^{-(t-s)P_0}\partial_{v_j }e^{-sP_0} = -\f 1 2 f'(s).
\]

To prove (\ref{x1}), one can check the following commutator relation: 
\bea \label{x2}
[ e^{-tP_0}, x_j] &=& - \int_0^t e^{(t-s)P_0} v_j e^{-s P_0} ds \\
& =& -e^{-tP_0} \left( v_j \sinh t   - 2 (\cosh t -1)  \partial_{v_j} + 2 (\sinh t -t) \partial_{x_j}\right), \nonumber
\eea
which means that the commutator initially defined as forms on $\vS\times \vS$ extends to operators from $\vS$ to  $L^2$ and the equality (\ref{x2}) holds. A successive application of this commutator relation shows that if $u\in \vS$, then $\w{x}^re^{-tP_0}u \in L^2$ for any $r \in \bR$.  It follows from (\ref{v1}) that 
\bea
e^{-(t-s)P_0}x_j e^{-sP_0} &=& e^{-tP_0}x_j + \int_0^s e^{-(t-\tau) P_0}v_j e^{-\tau P_0} \; d \tau  \nonumber \\
 &= & e^{-tP_0}(x_j+ v_j \sinh s - 2 (\cosh s -1)\partial_{v_j}  + 2 (\sinh s -s) \partial_{x_j}). \nonumber
\eea
This proves (\ref{x1}).
\ef

\begin{prop} \label{prop3.11} Let $n=3$.  Assume that $u \in \vL^{2, -s}$ for some $ \f 3 2 < s < 2$ such that  there exists some constant $c_0$  and $\psi\in L^2_v$ with  $(-\Delta_v + v^2)\psi \in L_v^2$ such that
\[
u(x, v) - c_0 (1\otimes \psi)  \in \vL^{2, \delta}(\bR_{x,v}^{6})
\]
for some $\delta>0$. Then on has
\be
\lim_{\lambda \to 0_-}\lambda R_0(\lambda) u = - c_0 \w{\psi_0, \psi}_{L^2_v}\gm_0
\ee
in $\vL^{2,-s}$ for any $s > \f 3 2$, where $\gm_0 = 1\otimes \psi_0(v)$.
\end{prop}
\pf For $\lambda <0$, $R_0(\lambda)$ maps $\vL^{2,-s}$ to $\vL^{2,-s}$ for any $s$ and one has
\[
\|R_0(\lambda)\|_{\vB(0,0; 0,0)} \le \frac{C}{|\lambda|}, \quad  \|R_0(\lambda)\|_{\vB(0,s'; 0,-s)} \le C_{s,s'}
\]
if $s,s' >\f 1 2$ with $s+s'>2$, uniformly in $\lambda \in ]-1, 0[$. An argument of complex interpolation shows that
for any $s, s' >0$, there exists $\ep>0$ such that
\be
 \|R_0(\lambda)\|_{\vB(0,s'; 0,-s)} \le C |\lambda|^{-1 +\ep}.
\ee
This shows that  
\be \label{eq2.72}
\lambda R_0(\lambda) (u - c_0 (1\otimes \psi)) = o(1), \quad \mbox{ as } \lambda \to 0_-
\ee
in $\vL^{2,-s}$ for any $s>0$. To prove Proposition \ref{prop3.11}, it suffices to study the limit $\lim_{\lambda \to 0_-}\lambda R_0(\lambda) (1\otimes \psi)$. By Proposition \ref{prop3.5}, the resolvent $R_0(\lambda)$ can be decomposed as
\be
R_0(\lambda) = b_0^w(v,D_x,D_v) (-\Delta_x- \lambda)^{-1} + r_0(\lambda)
\ee 
where 
\[
b_0(v, \xi, \eta) = 2^{\f 3 2}\chi(\xi)  e^{-v^2-\eta^2 + 2i v\cdot\xi + 2\xi^2}
\]
with $\chi$  a smooth cut-off around $0$ with compact support, and $r_0(\lambda)$ is uniformly bounded  as operators in $L^2$ for $\lambda < a$ for some $a\in ]0, 1[$. \\

We claim that the following estimate holds uniformly in $\lambda <a$:
\be \label{eq3.51}
\|\w{x}^{-2} r_0(\lambda) \w{x}^{2} f\| \le C (\|f\| + \|H_0 f\|)
\ee
for any $f \in D(H_0 ) $,  where $H_0 = -\Delta_v + v^2 -\Delta_x$.
 Remark that $r_0(\lambda)$ is a pseudodifferential operator in $x$-variables: $r_0(\lambda) = r_0(\lambda, D_x)$, with operator-valued symbol  $r_0(\lambda, \xi)\in \vB(L^2_v)$ (see (\ref{R00})). The proof of Proposition \ref{prop3.5} shows that $r_0(\lambda, \xi)$ can be decomposed as
\[
r_0(\lambda, \xi) =r_{0,1}(\lambda,\xi) + r_{0,2}(\lambda, \xi)
\]
 with
\[
r_{0,1}(\lambda, \xi) = \int_{0}^T \chi_1(\xi) e^{-t(\hat{P}_0(\xi) -\lambda)} dt
\]
for some fixed $T \ge 3$ and $r_{0,2}(\lambda,\xi)$ is smooth and rapidly decreasing in $\xi$, uniformly for $\lambda <a$ (see (\ref{eq3.17}), (\ref{eq3.19}) and (\ref{eq3.20})).   Since $r_{0,2}(\lambda, D_x)$ is a convolution in $x$-variables with a smooth and rapidly decreasing kernel, one has $\w{x}^{-s} r_{0,2}(\lambda, D_x)\w{x}^s$ is uniformly bounded for any $s$.
 To study $
r_{0, 1}(\lambda, D_x) $, we use commutator techniques. One writes
\[
\w{x}^{-2} r_{0, 1}(\lambda, D_x) x^2 = \w{x}^{-2}  \left(x^2 r_{0, 1}(\lambda, D_x) + \sum_{j=1}^n [ \; [r_{0, 1}(\lambda, D_x), x_j], x_j]\right).
\]
Making use of  (\ref{x2}), one can calculate
\bea
\lefteqn{[r_{0,1}(\lambda, D_x), x_j]} \nonumber \\
&= & -i(\partial_{\xi_j}\chi_1)(D_x) \int_{0}^T  e^{-t(\hat{P}_0(\xi) -\lambda)} dt  +
\chi_1(D_x) \int_{0}^T  [ e^{-t(\hat{P}_0(\xi) -\lambda)}, x_j] dt    \nonumber \\
&=&  -i(\partial_{\xi_j}\chi_1)(D_x) \int_{0}^T  e^{-t(\hat{P}_0(\xi) -\lambda)} dt  \\
& & +
\chi_1(D_x) \int_{0}^T e^{-t(P_0-\lambda)} (v_j \sinh t - 2 (\cosh t-1)\partial_{v_j} 
 + 2 (\sinh t - t) \partial_{x_j}) dt  \nonumber
\eea
Since $P_0$ is accretive, one obtains
 \[
 \|[r_{0,1}(\lambda, D_x), x_j] f\| \le C (\|f\|+ \|v_jf\| + \|\partial_{v_j} f \| + \|\partial_{x_j}f\|), \quad f\in \vS. 
 \]
Similarly, one can check that the second commutator  $ [ \; [r_{0, 1}(\lambda, D_x), x_j], x_j]$ verifies
\[
\|  [ \; [r_{0, 1}(\lambda, D_x), x_j], x_j] f\| \le C(\|f\| + \|H_0 f\|).
\]
This proves that 
\be \label{eq3.53}
\|\w{x}^{-2} r_{0,1}(\lambda) \w{x}^{2} f\| \le C  (\|f\| + \|H_0 f\|)
\ee
for any $f \in D(H_0 ) $, which gives   
(\ref{eq3.51}). It follows from (\ref{eq3.51}) and the uniform boundedness of $\|r_0(\lambda)\|$ for $\lambda <a$, $a>0$,  that for any $ s\in[0, 2]$
\be
\|\w{x}^{-s} r_0(\lambda) \w{x}^{s} f\| \le C (\|f\| + \|H_0 f\|), 
\ee
Since $1\otimes\psi = \w{x}^s f_s$ with $f_s = \w{x}^{-s} \otimes\psi  \in D(H_0)$ for any $s>\f 3 2$, it follows that 
\be
 \lambda r_0(\lambda)(1\otimes\psi) = O(|\lambda|), \quad  \lambda \to 0,
 \ee \label{eq2.79}
  in $ \vL^{2,-s}$  for any $s> \f 3 2$, which together with (\ref{eq2.72}) implies that
 \be
 \lambda R_0(\lambda) u  - c_0 \lambda  b_0^w(v,D_x,D_v) (-\Delta_x- \lambda)^{-1} (1 \otimes \psi) = o(1)
 \ee
in $\vL^{2,-s}$, $s< \f 3 2$, as $\lambda \to 0_-$. This estimate is valid in all dimensions.\\

Finally we calculate  $ \lambda  b_0^w(v,D_x,D_v) (-\Delta_x- \lambda)^{-1} (1 \otimes \psi)$  which is  independent of $\lambda <0$ in dimension $n=3$.  In fact, according to (\ref{kernel}) one has for $\lambda <0$ and $n=3$
\bea 
\lefteqn{\lambda  b_0^w(v,D_x,D_v) (-\Delta_x- \lambda)^{-1} (1 \otimes \psi)}  \\
& =& \int_{\bR^9}\f{\lambda  e^{-\sqrt{|\lambda|}|y-(x-x')|}}{4\pi |y -(x-x')|} \Phi(v,v',y) \psi(v')\; dy dx' dv'
\nonumber \\
&= & \int_{\bR^3} \f{\lambda e^{-\sqrt{|\lambda|}|x'|}}{4\pi |x'|} dx'  \int_{\bR^6} \Phi(v,v',y) \psi(v')\; dy  dv'
\nonumber \\
&=& -  \int_{\bR^3_{v'}}  (\int_{\bR^3}\Phi(v,v',y) \; dy )\psi(v')\;  dv' \nonumber \\
&=&  - \chi(0) \w{\psi_0, \psi}_{L^2_v} \psi_0(v) = - \w{\psi_0, \psi}_{L^2_v} \psi_0(v).  \nonumber
\eea
This finishes the proof of Proposition \ref{prop3.11}.
\ef

\sect{Threshold spectral properties of the Kramers-Fokker-Planck operator}

Consider the Kramers-Fokker-Planck operator $ P=v\cdot\na_x-\na_x V(x)\cdot\na_v-\Dl_v+\f{1}{4}|v|^2-\f{n}{2}$ with a $C^1$ potential $V(x)$ satisfying
\begin{equation} \label{eq4.1}
|\nabla_x V(x)|\le C \w{x}^{-\rho-1}, \quad x\in \bR^n,
\end{equation}
 for some $\rho \ge -1$.  In fact, it is sufficient to suppose that $V(x)$ is a Lipschitz function satisfying (\ref{eq4.1}) about everywhere and one can even include some mild local singularities in $\nabla V(x)$, by using hypoellipiticiy of the operator. But we will not care about such generalization.
  $P$ defined on $D(P)= D(P_0)$ is maximally dissipative. By Proposition \ref{prop2.1},  if $\rho >-1$, $\na_x V(x)\cdot\na_v$ is  relatively compact with respect to $P_0$. Consequently, the spectrum of $P$ is discrete outside $\bR_+$ and the complex eigenvalues of $P$ 
may only accumulate towards points in $\bR_+$.  The thresholds of $P$ are the eigenvalues of $-\Dl_v+\f{1}{4}|v|^2-\f{n}{2}$ which are equal to $\bN$. To be simple, we study only the threshold zero in dimension $n=3$ under the condition $\rho>1$.

Denote $W =-\na_x V(x)\cdot\na_v$. One has $W^* = - W$ and
\[
(P-z)R_0(z) = 1 + R_0(z)W = 1+ G_0W + O(|z|^\ep), \quad \ep>0,
\]
 in $\vB(0, -s; 0,-s)$
for  $1<s < (1+\rho)/2$ and $z$ near $0$ and $z\not\in \bR_+$.

\begin{lemma}\label{lem4.1}  Assume $n=3$ and  let (\ref{eq4.1}) be satisfied with $\rho >1$ ({\it i.e.}, the potential is of short-range). Then, $G_0W$ is a compact operator in $\vL^{2,-s}$ for $1<s< (1+\rho)/2 $
and
\be \label{eq4.2}
\ker_{\vL^{2,-s}} (1 + G_0W) =\{0\}.
\ee
\end{lemma}
\pf For $1<s< (1+\rho)/2$, take $1<s' <s$.   Proposition \ref{prop3.7} (and the arguments used in the proof of Corollary \ref{cor3.6}) shows that $G_0W\in \vB(0,-s; 1, -s')$. The injection $\vH^{1, -s'}$ into $\vL^{2,-s}$ is compact when $s'<s$. Therefore $G_0W$ is a compact operator in $\vL^{2,-s}$. \\

Let $u \in \vL^{2,-s}$ with $u + G_0Wu =0$.
 Then by the hypoellipticity of $P_0$,  one can check that $u \in \vH^{2, -s}$ for any $s>1$ and $Pu=0$. According to (\ref{G0}),
 $u$ can be decomposed as
\[
u=- F_0 W u - F_1 W u . 
\]
 Since $Wu \in \vH^{-1, \rho +1-s}$ and $F_0 \in \vB(-1,s; 1, -s')$ for any $s, s'>\f 1 2$ and $s+s'>2$,  it follows that $u$ is in fact in $\vH^{2, -s}$ for any $s >\f 1 2$. Using the condition $\rho>1$ and repeating the above argument,  we deduce that  $F_1 W u \in L^2$. By (\ref{F0}), one can calculate the asymptotic behavior of  $F_0 W u$ for $|x|$ large and obtains that
\be \label{eq4.4}
u(x, v) =w(x,v)  + r(x, v),
\ee
 where $\w{v}^2r \in L^2(\bR^6_{x,v})$ and $w(x,v) = C(u)\frac{e^{-\f{v^2}{4}}}{|x|}$ with
\be \label{eq4.5}
 C(u) = \int\int_{\bR^6}  \frac{e^{-\frac{v^2 }{4}}}{2  (2\pi)^{\f 5 2} } \nabla_x V(x) \cdot\nabla_v u(x, v) \; dx dx. \ee
Let $\chi \in C_0^\infty(\bR)$ be a cut-off with $\chi(\tau) =1$ for $ |\tau| \le 1$ and $\chi(\tau) =0$ for $|\tau| \ge 2$ and $0 \le \chi(\tau) \le 1$. Set $\chi_R(x) = \chi(\frac{|x|}{R}), R>1$ and $x\in\bR^3$  and $u_R(x) = \chi_R(x) u(x, v)$.
Then one has
\[
P u_R = \frac{v\cdot\hat x}{R} \chi'(\frac{|x|}{R}) u.
\]
Taking the real part of the equality $\w{ P u_R, u_R}= \w{ \frac{v\cdot\hat x}{R} \chi'(\frac{|x|}{R}) u, u_R}$, one obtains
\be  \label{eq4.6}
\int\int_{\bR^6} |(\partial_v + \frac v 2)u(x,v)|^2 \chi(\frac{|x|}{R})^2 \; dx dv =  \w{ \frac{v\cdot\hat x}{R} \chi'(\frac{|x|}{R}) u, u_R}
\ee
Since $w$ is even in $v$ (see (\ref{eq4.4})),  one has 
\[
\w{ \frac{v\cdot\hat x}{R} \chi'(\frac{|x|}{R}) w, \chi_R w} = 0
\]
 and the right hand side of the  equality (\ref{eq4.6}) satisfies
\bea
\lefteqn{|\w{ \frac{v\cdot\hat x}{R} \chi'(\frac{|x|}{R}) u, u_R} |}  \\
&   =  & |2  \Re  \w{ \frac{v\cdot\hat x}{R} \chi'(\frac{|x|}{R}) w, \chi_R r} +  \w{ \frac{v\cdot\hat x}{R} \chi'(\frac{|x|}{R}) r, \chi_R r} |\nonumber \\
& \le   & C R^{-(1-s)}(\|w\|_{\vL^{2,-s}} + \|r\|_{L^2}) \|\w{v}r\|_{L^2} \nonumber
\eea
for some $\f 1 2 < s <1$.  Taking the limit $R\to +\infty$ in (\ref{eq4.6}), one obtains that $(\partial_v + \frac v 2)u(x,v) \in L^2$ and
\be
\int\int_{\bR^6} |(\partial_v + \frac v 2)u(x,v)|^2  \; dx dv = 0.
\ee
This shows that $(\partial_v + \frac v 2)u(x,v) =0$, a.e. in $x,v$. Since $u \in \vL^{2, -s}$ for any $s>\f 1 2$ and $Pu=0$, one sees that $u$ is of the form $u(x, v) = C(x) e^{-\frac{v^2}{4}}$ for some $C \in L^{2,-s}(\bR^3_x)$ verifying  the equation 
\be \label{eq4.71}
v\cdot\na_x C(x) + \f 1 2 v \cdot\na V(x) C(x)=0
\ee
 a.e. in $x$ for all $v \in \bR^3$.  This proves that $C(x) = c_0 e^{-\f{V(x)}{2}}$ a.e. for some constant $c_0$ and
\[
u(x, v) = c_0 \gm.
\]
 Since $u \in \vL^{2, -s}$ for any $s>\f 1 2$ and $\gm\not\in \vL^{2,-s}$ if $\f 1 2 <s < \f 3 2$ when  $V(x)$ is bounded, one concludes that $c_0 =0$, therefore $u=0$. This proves that 
$\ker_{\vL^{2,-s}} (1 + G_0W) =\{0\}$.
\ef

A solution, $u$,  of the stationary equation $Pu=0$ is called a resonant state if $u \in \vL^{2,-s}$ for any $s>\f 1 2$, but $u \not\in L^2$ and zero is then called a resonance of $P$.  Lemma \ref{lem4.1} shows that zero is neither an eigenvalue nor a resonance ({\it i.e.}, a regular point) of the KFP operator $P$ if the potential is of short-range. 
This is in sharp contrast to Schr\"odinger operators for which zero resonance may exist even for smooth and compactly supported potentials. Lemma \ref{lem4.1} makes easier the threshold spectral analysis for the KFP operator.

\begin{theorem} \label{th4.2}  Assume $n=3$ and $\rho>1$.
Then  zero is not an accumulation point of the eigenvalues of $P$ and one has for any $s, s'>\f 1 2$ with $s+s'>2$, $\exists \ep >0$ such that
\be \label{eq4.8}
R(z) = A_0 + O(|z|^\ep), \quad   z\to 0, z\not\in\bR_+,
\ee
in $\vB(-1, s; \; 1, -s')$, where 
\be
 A_0 = (1 + G_0W)^{-1} G_0.
\ee
There exists  $\delta>0$ such that boundary values of the resolvent
\[
R(\lambda\pm i0) = \lim_{\ep \to 0_+} R(\lambda\pm i\ep), \quad \lambda \in ]0, \delta]
\]
exist in $\vB(-1, s; \; 1, -s)$ for any $s>\f 1 2$ and are continuous in $\lambda$.
\end{theorem}
\pf Remark that $G_0 \in \vB(-1, s; 1, -s')$ for any $s, s' > \f 1 2$ and $s+ s' >2$. 
Lemma \ref{lem4.1} implies that $\ker_{\vH^{1,-s}} (1 + G_0W) =\{0\}$ and $ G_0 W$ is compact on $\vH^{1,-s}$ 
for $1<s <(1+\rho)/2$. Therefore $1 + G_0W$ is invertible on $ \vH^{1, -s}$  with bounded inverse.
Since 
\[
1+ R_0(z) W =  1+ G_0W + O(|z|^\ep), \quad \mbox{ in } \vB(1,-s; 1,-s), \quad \ep>0, 
\]
 for $|z|$ small and $z \not\in \bR_+$,   it follows that $1+ R_0(z) W$ is invertible  on $ \vH^{1, -s}$ for $|z|$ small and $z \not\in \bR_+$ and that
\be
(1+ R_0(z) W)^{-1}= (1+ G_0W)^{-1} + O(|z|^\ep), \mbox{ as } |z| \to 0, z \not\in \bR_+.
\ee
In particular, $1+ R_0(z) W$ is injective in $\vL^{2,-s}$ for $z$ near $0$ and $z \not\in \bR_+$. Since $R_0(z) (P-z) = 1+ R_0(z) W$, this shows that $P$ has no eigenvalues for $|z|<\delta$ for some $\delta >0$ and that
\be \label{eq4.9}
R(z) = (1+ R_0(z) W)^{-1}R_0(z) =  A_0 + O(|z|^\ep), z\not\in \bR_+,
\ee
with $A_0 = (1 + G_0W)^{-1} G_0$. The existence of the boundary values $R(\lambda\pm i0)$ for $0<\lambda <\delta$ follows from the first equality in (\ref{eq4.9}).
 \ef

\begin{theorem} \label{th4.3}  Assume $n=3$ and $\rho>2$.
Then  for any $s> \f 3 2$, there exists $\ep >0$ such that
\be \label{eq4.10}
R(z) = A_0 +  z ^{\f 1 2} A_1 + O(|z|^{\f 1 2 +\ep}), 
\ee
in $\vB(-1, s; \; 1, -s)$ for $   |z|$ small and $ z\not\in\bR_+$, where $A_1$ is an operator of rank one given by
\be
A_1 = (1+ G_0W)^{-1}G_1 ( 1- WA_0).
\ee
\end{theorem}
\pf
For $s>\f  3 2$, one has
\[
R_0(z) = G_0 +  z ^{\f 1 2} G_1 + O(|z|^{\f 1 2 +\ep})
\]
in $\vB(-1, s; 1, -s)$. If $\rho>2$, one has $W \in \vB(0, -r; 1, \rho+1-r)$. Therefore for  $\f 3 2 < s < (\rho+1)/2$, it 
\[
 R_0(z) W - (G_0 +  z ^{\f 1 2} G_1 )W =  O(|z|^{\f 1 2 +\ep})
\]
in $\vB(0, -s; 0, -s)$. Since $B_0 \triangleq (1+G_0W)^{-1} \in \vB(0, -s; 0, -s)$, it follows that
\[
 (1+ R_0(z)W)^{-1} = B_0 - z ^{\f 1 2}B_0 G_1 W B_0  +  O(|z|^{\f 1 2 +\ep}).
\]
From the resolvent equation $R(z) = (1+ R_0(z) W)^{-1}R_0(z)$, we obtain that
\[
R(z) = B_0G_0 +  z ^{\f 1 2} B_0G_1 (1- W B_0 G_0) +  O(|z|^{\f 1 2 +\ep})
\]
in $\vB(0, s; 0, -s)$. An argument of hypoellticity shows that the same asymptics holds in $\vB(-1, s; 1, -s)$. Remark that $A_1= B_0G_1 (1- W B_0 G_0)$ is a rank one operator,  because $G_1$ is  of  rank  one. \ef

\sect{Large-time behaviors of solutions }

The large-time behaviors of solutions to the KFP equation with a potential will be deduced from resolvent asumptotics and a representation formula of the semigroup $S(t) =e^{-tP}$ in terms of the resolvent. To this purpose, we need the following high energy pseudospectral estimate.

\begin{theorem}\label{th5.1} Let $n \ge 1$ and assume (\ref{eq4.1}) with $\rho \ge -1$. Then there exists $C>0$ such that   $\sigma (P) \cap \{z; |\Im z| > C,  \Re z \le \f 1 C |\Im z|^{\f 1 3}  \} = \emptyset$ and
\be \label{eq5.1}
\|R(z)\| \le \f{C}{|z|^{\f 1 3}}, \quad
\ee
and
\be \label{eq5.2}
\|(1-\Delta_v + v^2)^{\f 1 2}R(z)\| \le \f{C}{|z|^{\f 1 6}}, \quad
\ee
for $|\Im z| > C $ and $ \Re z \le \f 1 C |\Im z|^{\f 1 3}$.
\end{theorem}
\pf Let $ W = -\nabla_xV(x)\cdot \nabla_v$. (\ref{eq3.40b}) shows that $\|WR_0(z)\| + \|R_0(z) W \| =O( |z|^{-\f 1 6})$ for $z$ in the region $\{z; |\Im z| > M,  \Re z \le \f 1 M |\Im z|^{\f 1 3}  \}$. Therefore $(1+ R_0(z) W)^{-1}$ exists and is uniformly bounded if $\{z; |\Im z| > M,  \Re z \le \f 1 M |\Im z|^{\f 1 3}  \}$ with $M$ sufficiently large.
Theorem \ref{th5.1} follows from Proposition \ref{prop3.8} and the resolvent equation $R(z) = (1+ R_0(z) W)^{-1} R_0(z)$
for $|\Im z| > C $ and $ \Re z \le \f 1 C |\Im z|^{\f 13}$ with $C\ge M$ sufficiently large.
\ef

\begin{lemma}\label{lem5.2} Let $n \ge 1$ and assume (\ref{eq4.1}) with $\rho \ge -1$. Then 
\be \label{eq5.3}
S(t) f =  \frac{1}{2\pi i} \int_{\gamma } e^{-t z}R(z)f dz
\ee
for $f\in L^2$ and $t>0$, where  the contour $\gamma$ is chosen such that 
\[\gamma = \gamma_- \cup \gamma_0 \cup \gamma_+\]
with $\gamma_{\pm} =\{ z; z = \pm i C + \lambda \pm i C \lambda^3, \lambda \ge 0\}$ and $\gamma_0$ is a curve in the left-half complexe plane joining $-i C$ and $iC$ for some $C>0$ sufficiently large, $\gamma$ being oriented from $-i\infty$ to $+i\infty$. 
\end{lemma}
\pf The spectrum of $P$ is void in the left side of $\gamma$. By Theorem \ref{th5.1}, for $C>0$ sufficiently large, $\gamma$ is contained in the resolvent set of $P$ and
\[
\|R(z)\| \le  \f C {|z|^{\f 1 3}}, \quad z \in \gamma.
\]
 Therefore, the integral $\tilde S(t) = \frac{1}{2\pi i} \int_{\gamma } e^{-t z}R(z) dz$ is norm convergent.
 In addition, one can check as in the standard case (see, for example, \cite{Kat80}) that  $\tilde S'(t)f = -P\tilde S(t)f$
 for $f \in D(P_0)$ and that $\lim_{t \to 0+} \tilde S(t)=I$ strongly. The uniqueness of solution to the evolution equation  $u'(t) + Pu(t)=0$ for $t>0$ and $ u(0) =u_0$ implies that $\tilde S(t)= S(t)$, $t>0$.
\ef

\begin{cor}\label{cor5.3} Assume that $n=3$ and $\rho>1$. One has
\begin{equation} \label{eq5.5}
\w{S(t)f,g} =\frac{1}{2\pi i} \int_{\Gamma} e^{-t z}\w{R(z)f,g} dz, \quad t>0,
\end{equation}
for any $f,g \in \vL^{2, s}$, $s>1$. Here
\[
\Gamma = \Gamma_- \cup \Gamma_0 \cup \Gamma_+
\]
with $\Gamma_{\pm} =\{ z; z = \delta + \lambda \pm i  \delta^{-1} \lambda^3, \lambda \ge 0\}$ for $\delta>0$ small enough and $\Gamma_0 =\{ z = \lambda \pm i0; \lambda \in [0, \delta]\}$. $\Gamma$ is oriented from $-i\infty$ to $+i\infty$.
\end{cor}
\pf 
$P$ has no eigenvalues with real part equal to zero (see the arguments used in the proof of Lemma \ref{lem4.1}) and their only possible accumulation points are in $\bR_+$. If $n=3$ and $\rho>1$,  Theorem \ref{th4.2} and Theorem \ref{th5.1}
imply that there exists some $\delta_0 >0$ such that
\be
\sigma_p(P) \cap \{z; \Re z \le \delta_0\} = \emptyset.
\ee
Consequently if $0<\delta <\delta_0$ is small enough, there is no spectrum of $P$ in the interior of the region between $ \gamma$ and $\Gamma$ and the resolvent is holomorphic there. In addition, the limiting absorption principles at low energies ensure that the integral
\[
\int_{\Gamma} e^{-t z}\w{R(z)f,g} dz, \quad t>0,
\]
is convergent for any $f,g \in \vL^{2, s}$ with $s>1$. By analytic deformation, we conclude from (\ref{eq5.1}) that
\[
\w{S(t)f,g} = \frac{1}{2\pi i} \int_{\gamma} e^{-t z}\w{R(z)f,g} dz = \frac{1}{2\pi i} \int_{\Gamma} e^{-t z}\w{R(z)f,g} dz, \quad t>0,
\]
for any $f,g \in \vL^{2, s}$ with $s>1$.
\ef

The formula (\ref{eq5.5}) is useful for studying the time-decay of solutions to the KFP equation with a short-range potential (satisfying (\ref{ass}) with $\rho>1$).

\begin{theorem}\label{th5.4} Assume  $n=3$.
\\

 (a). If $\rho>1$, one has for any $s> \f 3 2$
\be \label{eq5.8}
\|S(t)\|_{\vB(0,s; 0, -s)} \le C_s  t^{- \f 3 2}, \quad t  >0.
\ee

(b). If  $\rho >2$, then  for any $s> \f 3 2$, there exists some $\ep >0$ such that
\be \label{eq5.7}
S(t) =  t^{- \f 3 2} B_1 + O(t^{- \f 3 2 -\ep})
\ee
in $\vB(0,s; 0, -s)$  as $t \to + \infty$, where
\be
B_1 = \frac{1}{2  i\sqrt \pi} A_1
\ee
is an operator of rank one. \\
\end{theorem}
\pf Assume that $n=3$  and $\rho>1$. By Corollary \ref{cor5.3}, one has for $f,g \in \vL^{2,s}$ with $s>1$ and for $t>0$
\bea
\w{S(t)f,g} &= & \frac{1}{2\pi i} \int_{0}^{\delta} e^{- t \lambda}\w{(R(\lambda +i0)-R(\lambda -i0)) f,g} d\lambda \nonumber  \\
&     + & \frac{e^{-t \delta}}{2\pi i} \int_{0}^{\infty} e^{-t (\lambda + i  \delta^{-1} \lambda^3)}\w{R(\delta + \lambda +  i  \delta^{-1} \lambda^3 ) f,g} ( 1 +  3 i \delta^{-1} \lambda^2)  d\lambda \\
 &  - &
 \frac{e^{ -t \delta}}{2\pi i} \int_{0}^{\infty} e^{-t (\lambda - i  \delta^{-1} \lambda^3)}\w{R(\delta + \lambda - i \delta^{-1} \lambda^3 ) f,g} ( 1 -  3  i \delta^{-1} \lambda^2)  d\lambda \nonumber
\\
& \triangleq & I_1+ I_2 + I_3. \nonumber
\eea

For $I_2$ and $I_3$, one can apply Theorem \ref{th5.1} to estimate 
\[
|\w{R(\delta + \lambda \pm i  \delta^{-1} \lambda^3 ) f,g}| \le C_M \|f\|_{\vH^{-1,s}}\|g\|_{\vH^{-1,s}}
\]
for $s> \f 12$ and $\lambda \in ]0, M]$ for each fixed $M>0$ and
\[
|\w{R(\delta + \lambda \pm i  \delta^{-1} \lambda^3 ) f,g}| \le C_M \lambda^{-\f 1 3}\|f\|_{L^2}\|g\|_{L^2}
\]
for $\lambda >M$ with $M>1$ sufficiently large. Therefore, if $\rho>1$
\be
|I_k| \le C e^{ -t \delta} \|f\|_{\vH^{0,s}}\|g\|_{\vH^{0,s}}
\ee
for  $k =2,3$ and for any $s > \f 1 2$. \\

To show (\ref{eq5.8}), it remains  to prove that if $\rho>1$ and $n=3$, one has
\be \label{eq5.11}
|\int_{\Gamma_0} e^{- t z}  \w{ R(z)f,g} dz |  \le C t^{-\f 3 2}\|f\|_{\vL^{2,s}}\|g\|_{\vL^{2,s}}
\ee
for any $f, g \in \vL^{2,s}$, $s>\f 3 2$. For $ 1 < s < (\rho+1)/2$, one has for some $\ep_0>0$ 
\[
W (R_0(z) - G_0) = O(|z|^{\ep_0}),  \quad \mbox{ in } \vB(0, s; 0, s) 
\]
for $z$ near $0$ and $z \not\in\bR_+$. By Lemma \ref{lem4.1}, $1+ WG_0$ is invertible in  $\vB(0, s; 0, s) $ for any
$1<s < (\rho+1)/2$. One obtains
\be \label{eq5.12}
R(z) = R_0(z)(1+ W R_0(z))^{-1} =  \sum_{j=0}^N  R_0(z) T(z)^j (1+WG_0)^{-1} +  O(|z|^{(N+1)\ep_0})
\ee
in $\vB(0, s; 0, s) $ with $s>1$, where $N$ is taken such that $(N+1)\ep_0 > \f 12$ and 
\be
T(z) = (1+WG_0)^{-1} W (R_0(z) -G_0).
\ee
Consequently
\be \label{eq5.13}
|\int_{\Gamma_0} e^{- t z}  \w{ (R(z)- R_0(z) \sum_{j=0}^N  T(z)^j (1+WG_0)^{-1})f,g} dz |  \le C t^{-\f 3 2 - \ep}\|f\|_{\vL^{2,s}}\|g\|_{\vL^{2,s}}, \quad \ep >0,
\ee
if $s>1$. (\ref{eq5.11}) follows from the following lemma which achieves the proof of (\ref{eq5.8}). 
\ef

\begin{lemma} \label{lem5.5} For each $j\ge 0$ and $\f 3 2 < s < \rho + \f 12$, there exists some  $C>0$ such that
\be \label{eq5.14}
|\int_{\Gamma_0} e^{- t z}  \w{R_0(z) T(z)^j (1+WG_0)^{-1}f,g} dz |  \le C t^{-\f 3 2}\|f\|_{\vL^{2,s}}\|g\|_{\vL^{2,s}}
\ee
for  any $f, g \in \vL^{2,s}$ and $t >0$.
\end{lemma}
\pf  We want to show that
\be \label{eq5.15}
| \w{ (R_0(z) T(z)^j - R_0(\overline{z}) T(\overline{z})^j)  (1+WG_0)^{-1}f,g}  |  \le C \sqrt \lambda \|f\|_{\vL^{2,s}}\|g\|_{\vL^{2,s}}
\ee
for $ z = \lambda + i0$ with  $\lambda \in [0, \delta]$ and for  $f, g \in \vL^{2,s}$ with  $ \f 3 2 < s < \rho + \f 12$.
Recall that $R_0(z) = b_0^w(v, D_x, D_v) (-\Delta_x -z)^{-1} + r_0(z)$ with $r_0(z)$ a bounded operator-valued function holomorphic in $z$ when $\Re z < a$  for some $0 <a <1$ (see (\ref{eq3.31})). Without loss, we can choose $\delta>0$ small such that $0<\delta <a$ which gives 
\be
R_0(z) - R_0(\bar z) = b_0^w(v, D_x, D_v) ( (-\Delta_x -z )^{-1} - (-\Delta_x -\bar z )^{-1})
\ee
for $ z = \lambda + i0$ with  $\lambda \in [0, \delta]$, $ 0<\delta <1$. Making use of the explicit formula for the integral kernel of
$(-\Delta_x -z )^{-1} $, one obtains
\be \label{eq5.16}
R_0(z) - R_0(\bar z)  = O(\sqrt \lambda), \quad \lambda \in [0, \delta]
\ee
in $\vB(-1, s; \; 1, -s)$ for any $s> \f 3 2$. By Proposition \ref{prop3.7},   
$R_0(z)$ is uniformly bounded in $\vB(-1,s; 1, -s')$ for $s, s' > \f 12$ and $s+s'>2$. We deduce that for any $ \f 12 < s < \rho + \f 12$,  $WG_0$ and
$(1+ WG_0)^{-1}$ belongs to $ \vB(0, s; \; 0,s)$  and 
\be
 T(z) = O(|z|^{\ep_0}) \quad   \mbox{ in  }\vB(0, s; 0, s).
 \ee
Seeing (\ref{eq5.16}), one obtains that for any $ \f 3 2 < s <\rho + \f 12$ and $s'= 1+ \rho -s > \f 1 2$,
\be T(z) - T(\bar z) = (1+ WG_0)^{-1} W( R_0(z) - R_0(\bar z)) = O(\sqrt \lambda )
\ee
in $\vB(0, s; \; 0,s')$  for  $ z = \lambda + i0$. Since for $j \ge 1$,
\bea
\lefteqn{R_0(z) T(z)^j - R_0(\overline{z}) T(\overline{z})^j }\\
&= & (R_0(z)- R_0(\bar z)) T(z)^j  + R_0(\overline{z}) \sum_{k=0}^{j-1}T(z)^k (T(z)- T(\overline{z})) T(\overline{z})^{j-k-1} \nonumber
\eea
one can estimate $|\w{ (R_0(z) T(z)^j - R_0(\overline{z}) T(\overline{z})^j)  (1+WG_0)^{-1}f,g}|$ by
\bea
\lefteqn{ |\w{ (R_0(z) T(z)^j - R_0(\overline{z}) T(\overline{z})^j)  (1+WG_0)^{-1}f,g}|} \nonumber  \\[2mm]
& \le & C \{ \|(R_0(z)- R_0(\bar z))\|_{\vB(0, s; \; 0,-s)} \|T(z)\|^j _{\vB(0, s; \; 0,s)} \\[2mm]
  & &  +   \sum_{k=0}^{j-1} \|R_0(\overline{z})\|_{\vB(0, s'; \; 0,-s)} \|(T(z)- T(\overline{z}))\|_{\vB(0, s; \; 0,s')} 
  \|T(z)\|^{k} _{\vB(0, s'; \; 0,s')} \|T(\overline{z})\|^{j-k-1} _{\vB(0, s; \; 0,s)} \}  \nonumber  \\[2mm]
   & & \times \|f\|_{\vL^{2,s}}\|g\|_{\vL^{2,s}}  \nonumber \\[2mm]
&\le &   C' \sqrt \lambda \|f\|_{\vL^{2,s}}\|g\|_{\vL^{2,s}}  \nonumber
\eea
for  $ \f 3 2 < s < \rho + \f 12$ and $s' = 1+ \rho -s >\f 1 2$. 
 The above estimate is  clearly  also true when $j=0$. 
This proves (\ref{eq5.15}) which implies (\ref{eq5.14}) and consequently  (\ref{eq5.8}). \ef

Assume now $\rho>2$.  Theorem \ref{th4.3} gives that if $\rho>2$,
\be \label{eq5.9}
R(z) = A_0 + \sqrt z A_1 + O(|z|^{\f 1 2 +\ep}), \quad   z\to 0, z\not\in\bR_+,
\ee
in $\vB(-1, s; \; 1, -s)$, $s> \f 3 2$. 
According to (\ref{eq5.9}), $I_1$ can be evaluated by
\[
I_1 = \frac{1}{\pi i} \int_{0}^{\delta} e^{- t \lambda}  \w{ (  \sqrt \lambda  A_1 + O(\lambda^{\f 1 2 +\ep}))f,g} d\lambda
\]
and
\be \label{eq5.10}
|I_1 - \frac{1}{  t^{\f 3 2}}\w{B_1 f,g}| \le C t^{-(\f 3 2 + \ep)}\|f\|_{\vH^{-1,s}}\|g\|_{\vH^{-1,s}}
\ee
with 
\be
 B_1 = \f{1}{\pi i} A_1 \int_0^\infty e^{-s} \sqrt{s} \; ds = \frac{1}{2  i\sqrt \pi} A_1.
\ee
This proves (\ref{eq5.7}).
\ef

\bigskip

{\noindent \bf Proof of Theorem \ref{th1.1}.}  Theorem \ref{th1.1} (a) is just Theorem \ref{th5.1},   Theorem \ref{th1.1} (b) is contained in  Theorem \ref{th4.2} and Theorem  \ref{th5.4} (a). Seeing Theorem \ref{th5.4} (b), to prove Theorem  \ref{th1.1} (c), it remains to calculate the operator $B_1$ given in Theorem \ref{th5.4}.  \\

For $u \in \vL^{2, s}$ with $s> \f 3 2$, one can write
\bea
B_1 u &=& \f{1}{8 \pi^{\f 3 2}}\w{ \gm_0,  (1- W A_0)u}  (1+ G_0 W)^{-1} \gm_0 \\
& =&  \f{1}{8 \pi^{\f 3 2}}\w{ (1- WA_0)^*\gm_0,  u}  (1+ G_0 W)^{-1} \gm_0 \nonumber \\
&= & \f{1}{8 \pi^{\f 3 2}} \w{\nu_0, u} \mu_0 \nonumber
\eea
where
\be
\mu_0 = (1+ G_0 W)^{-1} \gm_0, \quad  \nu_0 =   (1- WA_0)^*\gm_0.
 \ee
Since $\gm_0 \in \vL^{2,-s}$ for any $s> \f 3 2$ and $(1 + G_0 W)^{-1} \in \vB(0,-s; 0,-s)$,   $G_0 W$ and $(WA_0)^{*} \in \vB(0, -s,; 0, -s')$ for any $s'> \f 1 2$ and $\f 3 2 < s  <\f{\rho+1}{2}$ (if $\rho >2$),
$\nu_0, \mu_0$ belong to $\vL^{2,-s}$ for any $s>\f 3 2$. In addition, $\mu_0$ satisfies the equation
\[
\lim_{z\to 0, z\not\in \bR_+} ( 1 + R_0(z) W)\mu_0  = (1+ G_0W) \mu_0 = \gm_0
\]
in $\vL^{2,-s}$. It follows that 
\be
 P\mu_0 =P_0 \gm_0 =0. 
 \ee
Similarly, since $A_0 =\lim_{z\to 0, z\not \in \bR_+} R(z)$ in $\vB(-1, s; 1, -s')$ for any $s, s' > \f 1 2$ with $s+s'>2$ and $ (1- WA_0)^* = 1 + A_0^*W$ ($W$ being skew-adjoint), one can check that
\bea
P^*\nu_0 &=  &   (P_0+ W)^* + W)\gm_0 \\
&=  &   P_0^*\gm_0 =0. \nonumber
\eea

To prove that $\mu_0 = \gm$, we remark that the solution of the equation
$(1+ G_0 W)\mu = \gm_0$ is unique  in $\vL^{2, -s}$ for any $s> \f 3 2$. Therefore it suffices to check that
$\gm$ also verifies the equation $(1+ G_0 W)\gm =  \gm_0$.  To show this, we notice that since $(P_0 + W)\gm =0$, one has for $\lambda <0$
\be \label{eq5.29}
 ( 1+ R_0(\lambda) W) \gm = - \lambda R_0(\lambda) \gm
\ee 
For $\rho >2$, one has $\gm - \gm_0 \in \vL^{2, s}$ for any $0<s < \rho -\f 3 2$. Proposition  \ref{prop3.11} with $\psi =\psi_0$ shows that
\be
\lim_{\lambda \to 0 _-} \lambda R_0(\lambda) \gm = -\gm_0.
\ee
Taking the limit $\lambda \to 0_-$ in (\ref{eq5.29}), one obtains $(1+G_0W) \gm =  \gm_0 = (1+G_0W) \mu_0$.
According to Lemma \ref{lem4.1} with $\rho>2$, the operator $1 + G_0 W$ is invertible  in $\vL^{2,-s}$ for any $ \f 3 2 < s < \f{1+\rho}2$, which gives $\mu_0 = \gm = (1 + G_0 W)^{-1} \gm_0$. \\

To show that $\nu_0=  \gm$, we notice that $1+WG_0\in \vB(0, s; 0, s)$ for any $ \f 3 2 < s <\f{\rho+ 1}{2} $ and is invertible and its inverse is given by:
\be
( 1+ WG_0)^{-1} = 1 - W (1+G_0W)^{-1} G_0 = 1-W A_0.  
\ee
 Therefore $\nu_0 =(1-WA_0)^*\gm_0 = (1-G_0^*W)^{-1}\gm_0$. Since $\gm$ verifies also the equation $P^*\gm =(P_0^* -W)\gm =0$, the similar arguments as those used above allow to conclude that $ (1- G_0^* W)\gm = \gm_0$ which
shows $\nu_0  = (1-G_0^*W)^{-1}\gm_0 = \gm$. This shows
\be
B_1u = \f{1}{8\pi^{\f 3 2}}\w{\gm, u} \gm \quad \mbox{ for } u \in \vL^{2,s} \mbox{ with } s>\f 3 2,
\ee
which proves (\ref{eq1.24})   of Theorem \ref{th1.1}. \ef

In the proof of Theorem \ref{th1.1} (c), we showed that solutions to the equation $Pu =0$ with $u \in \vL^{2,-s}$ for any $s>\f 3 2$, are given by $u =c \gm$ for some constant $c$.  If $V$ is smooth and  coercive ($|\nabla V(x)|\to \infty$ and $V(x)>0$ outside some compact), the hypoelliptic estimate for $P$ allows to conclude that if $Pu=0$ and $u \in \vS'$, then $u \in \vS$ and $u = c \gm$ for some constant $c$. See \cite{hln,hss}. This kind of uniqueness result seems  to be unknown for potentials whose gradient tends to zero. Our proof not only shows that $\mu_0 = c \gm$, but also  compute the constant $c$ which allows to give the universal constant in the leading term of (\ref{eq1.24}.  \\

\noindent {\bf Remarks.} (a). Several interesting questions remain open. In particular, we do not know if results like (\ref{eq1.23}) and (\ref{eq1.24}) hold for $S(t) =e^{-t P}$  as operators from $\vL^1$ to $\vL^\infty$. See Theorem \ref{th3.3} for the free KFP operator.\\ 

(b). The assumptions on dimension and on the decay rate of the potential are only used in low-energy resolvent asymptotics. While we believe that the condition $n=3$ is only technical,  the condition on the decay rate $\rho$ is more essential to our approach which consists in regarding the free KFP operator $P_0$ as  model operator for the full KFP operator with a potential. To study the low-energy resolvent asymptotics for potentials with more slowly decreasing gradients, one may try to use other models such as  the Witten Laplacian 
\[
-\Delta_V =  (-\nabla_x + \nabla V(x))\cdot (\nabla_x + \nabla V(x))  
\]
See \cite{hkn,hln,hhs} for relations between  the KFP operator and the Witten Laplacian in eigenvalue problems, when $|\nabla V(x)| \to +\infty$ and the spectrum  is discrete near $0$. 
When $|\nabla V(x)|$ is slowly decreasing, under some reasonable additional conditions  $-\Delta_V $ is a Schr\"odinger operator $-\Delta_x + U(x)$ with a potential $U(x)$ positive outside some compact and slowly decreasing at the infinity. The threshold spectral properties  for this class of selfadjoint elliptic operators can be  fairly well analyzed, making use of known results on Schr\"odinger operators with globally positive and slowly decreasing potentials (\cite{nak,yaf}).  The open question here is to  see to which extent  the approximation of the KFP operator by the Witten Laplacian is valid in  a scattering framework. We hope to return to this problem in a forthcoming work.

\appendix

\sect{A family of complex harmonic oscillators}

In this appendix, we study some basic spectral properties of a family of non-selfadjoint harmonic oscillators 
\be
\hat{P}_0(\xi) = -\Dl_v+\f{v^2}{4} -\f{n}{2}+ i v\cdot\xi,
\ee 
where $\xi \in \bR^n$ are regarded as parameters. By Fourier transform in $x$-variables, the free KFP
$P_0$ is a direct integral of the family $\{\hat{P}_0(\xi) ; \xi \in \bR^n\}$. We give here some quantitative  results with explicit bounds in $\xi$.  Note 
that non-selfadjoint harmonic operators with complex frequency $ - \Delta_x + \omega x^2$, $\omega\in \bC$,  are studied by several authors. See for example \cite{boul, DK04} and references quoted therein.
\\

The operator $\hat{P}_0(\xi)$ can be written as 
\[
\hat{P}_0(\xi) =-\Dl_v+\f{1}{4}\sum^n_{j=1}(v_j+2i\xi_j)^2-\f{n}{2}+|\xi|^2.
\]
 $\{\hat{P}_0(\xi), \xi\in \bR^n\}$ is  a holomorphic family of type $(A)$ in sense of Kato with constant domain
 $D= D(-\Dl_v+\f{v^2}{4})$ in $L^2(\bR^n_v)$.
 Let $F_j(s)=(-1)^je^{\f{s^2}{2}}\f{d^j}{ds^j}e^{-\f{s^2}{2}}, j \in \bN,$ be the Hermite polynomials and
$$
\varphi_j(s)=(j!\sqrt{2\pi})^{-\f{1}{2}}e^{-\f{s^2}{4}}F_j(s)
$$
the normalized Hermite functions.  For $\xi \in \bR^n$ and $\alpha=(\alpha_1, \alpha_2, \cdots, \alpha_n) \in \bN^n$, define
\be
\psi_\alpha(v) = \prod_{j=1}^n\varphi_{\alpha_j}(v_j) \mbox{ and } \psi_\alpha^\xi(v) = \psi_\alpha(v + 2i \xi).
\ee
For $\alpha, \beta \in \bN^n$,  $\xi \to \w{\psi_\alpha^\xi, \psi_\beta^{-\xi}} $ extends to an entire function for $\xi \in \bC$ and is constant on $i\bR$. Therefore $\w{\psi_\alpha^\xi, \psi_\beta^{-\xi}}$ is  constant for $\xi \in \bC$ and one has
\be \label{basis}
\w{\psi_\alpha^\xi, \psi_\beta^{-\xi}} = \delta_{\alpha\beta} = \left\{\begin{array}{ll}
$1$, & \hbox{$\alpha =\beta$,} \\
$0$, & \hbox{$\alpha \neq \beta$.}
\end{array}
\right., \quad \forall  \alpha, \beta \in \bN^n, \xi \in \bR^n.
\ee
 Using the definition of Hermite functions, one can check that for $\alpha =(\alpha_1, \cdots, \alpha_n) \in \bN^n$
 \be\label{norm} \|\psi_\alpha^\xi\|^2=e^{2\xi^2}\prod_{m=1}^n(\sum_{j=1}^{\alpha_m}\f{C_{\alpha_m}^j}{j!}(2\xi_m)^{2j}). \ee
In fact, when $n=1$, $\xi =\xi_1$ and $k \in \bN$, one has
\begin{eqnarray*}
\|\psi_k^\xi\|^2&=&\int_\bR\varphi_k(y+2i\xi)\varphi_k(y-2i\xi)dy\\
&=&(k!\sqrt{2\pi})^{-1} e^{2\xi^2}\int_\bR F_k(y+2i\xi)F_k(y-2i\xi)e^{-\f{y^2}{2}}dy\\
&=&(k!\sqrt{2\pi})^{-1}e^{2\xi^2}\int_\bR (\sum_{j=0}^k C_k^j(2i\xi)^{k-j}F_j(y))\sum_{l=0}^k C_k^l(-2i\xi)^{k-l}F_l(y))e^{-\f{y^2}{2}}dy\\
&=&e^{2\xi^2}\sum_{j,l=0}^k(k!\sqrt{2\pi})^{-1}C_k^jC_k^l(j!\sqrt{2\pi})^{\f{1}{2}}(l!\sqrt{2\pi})^{\f{1}{2}}(2i\xi)^{k-j}(-2i\xi)^{k-l}\int_\bR\psi_j(y)\psi_l(y)dy\\
&=&e^{2\xi^2}\sum_{j=0}^k\f{j!(C_k^j)^2}{k!}(4\xi^2)^{k-j}\\
&=&e^{2\xi^2}\sum_{j=0}^k\f{C_k^j}{j!}(4\xi^2)^j.
\end{eqnarray*}
The general case $n \ge 1$ follows from the product formula: 
\[
\|\psi_\alpha^\xi\|^2 = \prod_{j=1}^n\|\varphi_{\alpha_j}(\cdot+2 i \xi_j)\|^2.
\]
 (\ref{norm})  shows that if $\xi \neq 0$, $\|\psi_\alpha^\xi\|$ grows exponentially as $|\alpha| \to +\infty$.

\begin{lemma}\label{lem2.2}
The spectrum of $\hat{P}_0(\xi)$ is purely discrete:
\be
\sigma (\hat{P}_0(\xi)) =\{E_l \triangleq l +\xi^2; l \in \bN\}.
\ee
Each eigenvalue $E_l$ is semi-simple ({i.e.}, its algebraic multiplicity and geometric multiplicity are equal) with multiplicity $m_l = \# \{\alpha\in \bN^n; |\alpha| =\alpha_1 + \alpha_2 + \cdots + \alpha_n = l\}$.
The Riesz projection associated with the eigenvalue $l + \xi^2$ is given by
 \be\label{projection} 
 \Pi^\xi_l\phi=\sum_{\alpha, |\alpha|=l} \langle\psi^{-\xi}_\alpha , \phi \rangle\psi^\xi_\alpha, \quad \phi \in L^2.
  \ee
\end{lemma}
\begin{proof} It is clear that the spectrum of $\hat{P}_0(\xi)$ is purely discrete and $\psi_\alpha^\xi(v)$ is an eigenfunction  associated with the eigenvalue $E_l$. This means that $\sigma (\hat{P}_0(\xi)) \supset \{ l +\xi^2; l \in \bN\}$. Since $\hat{P}_0(\xi)^* = \hat{P}_0(-\xi)$  and that the linear span of $\{\psi_\alpha^{-\xi}(v); \alpha\in \bN^n\}$ is dense in $L^2$,
$\hat{P}_0(-\xi)$ can not have other eigenvalues than $E_l$, $ l=0, 1, \cdots$. This proves that
$\sigma (\hat{P}_0(\xi)) = \{ E_l= l +\xi^2; l \in \bN\}$.

To show that $E_l$ is semisimple, assume by contradiction that $ \exists \varphi \in D$ such that
$(\hat{P}_0(\xi)- E_l) \varphi = c\psi_\alpha^{\xi}$, $|\alpha| =l$ and $c\in \bC^*$.  Then $\psi_\alpha^{\xi}$ is in the range of
$ \hat{P}_0(\xi)- E_l$ and hence is orthogonal to the kernel of
$ (\hat{P}_0(\xi)- E_l)^* = \hat{P}_0(-\xi)- E_l$. In particular, one has
\[
\w{\psi_\alpha^\xi, \psi_\alpha^{-\xi}} =0.
\]
This is impossible due to  (\ref{basis}). This contradiction shows that the eigenvalue $E_l = l +\xi^2$ is semisimple.
Since the pole of the resolvent $(\hat{P}_0(\xi)-z)^{-1}$ is simple at $z=E_l$, the range of the associated Riesz projection defined by
\[
\Pi^\xi_l \phi =\frac{i}{2\pi} \int_{|z-E_l| = \f 1 2} (\hat{P}_0(\xi)-z)^{-1}\phi  dz
\]
is equal to $\ker (\hat{P}_0(\xi))-E_l)$ and the kernel of $\Pi^\xi_l $ is equal to the range of  $\hat{P}_0(\xi)-E_l$. The latter is the orthogonal complement of $\ker  (\hat{P}_0(-\xi))-E_l)$ which is spanned by $\{ \psi_\alpha^{-\xi}; |\alpha|=l\}$. The representation formula (\ref{projection}) then follows from  (\ref{basis}).
\end{proof}

\begin{lemma}\label{lem2.3} Let $n=1$. Then one has for $t>0$
\be \label{1d}
\sum_{k=0}^{\i}e^{-t(k+\xi^2) }\| \Pi_k^\xi\| = \f{e^{-\xi^2(t - 2)}}{1-e^{-t}}e^{\f{4\xi^2}{e^t-1}}, \quad \xi \in \bR.
\ee
\end{lemma}
\pf
In the case $n=1$, one has for any $k \in \bN$
\be \label{norme}
\|\Pi^\xi_k\|=\|\psi^\xi_k\|\|\psi^{-\xi}_k\|=\|\psi_k^\xi\|^2=e^{2\xi^2}\sum_{j=0}^k\f{C_k^j}{j!}(4\xi^2)^j.
\ee

The left-hand side of (\ref{1d}) is norm convergent when $t>0$. In fact, one can calculate the sum of the series as follows 

\begin{eqnarray*}
\lefteqn{\sum_{k=0}^{\i}e^{-t(k+\xi^2) }\| \Pi_k^\xi\|} \\
&=& \sum_{k=0}^{\i}e^{-t(k+\xi^2) + 2\xi^2}\sum_{j=0}^{k}\f{C_k^j}{j!}(4\xi^2)^j\\
&=&\sum_{k=0}^{\i}e^{-t(k+\xi^2)+ 2\xi^2}+\sum_{k=1}^{\i}e^{-t(k+\xi^2) + 2\xi^2}\sum_{j=1}^{k}\f{C_k^j}{j!}(4\xi^2)^j\\
&=&\f{e^{-\xi^2(t - 2)}}{1-e^{-t}}+e^{-\xi^2(t - 2)}\sum_{j=1}^\i\f{(4\xi^2)^j}{j!}\sum_{k=j}^\i C_k^je^{-tk}\\
&=&\f{e^{-\xi^2(t - 2)}}{1-e^{-t}}+e^{-\xi^2(t - 2)}\sum_{j=1}^\i\f{(4\xi^2e^{-t})^j}{(j!)^2}\sum_{k=j}^\i k(k-1)\cdots(k-j+1)e^{-t(k-j)}\\
&=&\f{e^{-\xi^2(t - 2)}}{1-e^{-t}}+e^{-\xi^2(t - 2)}\sum_{j=1}^\i\f{(4\xi^2e^{-t})^j}{(j!)^2}\sum_{k=0}^\i(k+j)(k+j-1)\cdots(k+1)e^{-tk}\\
&=&\f{e^{-\xi^2(t - 2)}}{1-e^{-t}}+e^{-\xi^2(t - 2)}\sum_{j=1}^\i\f{(4\xi^2e^{-t})^j}{(j!)^2}\left\{\f{d^j}{dx^j}\f{1}{1-x}\right\}\mid_{x=e^{-t}}\\
&=&\f{e^{-\xi^2(t - 2)}}{1-e^{-t}}\left(1+\sum_{j=1}^\i\f{1}{j!}(\f{4\xi^2}{e^t-1})^j\right)\\
&=&\f{e^{-\xi^2(t - 2)}}{1-e^{-t}}e^{\f{4\xi^2}{e^t-1}}.
\end{eqnarray*}
\ef

\begin{prop}\label{prop2.3}  Let $ n\ge 1$. 
For any $\xi\in \bR^n$ and $t>0$, one has the following formula of spectral decomposition:
\be\label{expansion} 
e^{-t \hat{P}_0(\xi) }=\sum_{l=0}^{\i}e^{-t(l+\xi^2)}\Pi_l^\xi, 
\ee
where $\Pi_l^\xi$ is the Riesz projection associated with the eigenvalue $l$ of $\hat{P}_0(\xi)$  and the series is norm convergent as operators on $L^2(\bR^n_v)$.
\end{prop}
\begin{proof} For $n \ge 1$, one has
\[
\|\Pi_l^\xi\| \le \sum_{\alpha= (\alpha_1, \cdots, \alpha_n)\in \bN^n; |\alpha| =l} \prod_{j=1}^n \|\varphi_{\alpha_j}(\cdot+ 2i\xi_{j})\|^2.
\]
By Lemma \ref{lem2.3}, the right-hand side of (\ref{expansion}) is norm convergent for every $t>0$  and can be evaluated by
\bea \label{eq2.10}
\sum_{l=0}^{\i}e^{-t(l+\xi^2)}\|\Pi_l^\xi\| 
& \le & \prod_{j=1}^n \left(\sum_{\alpha_j=0}^\infty e^{-t(\alpha_j+\xi_j^2) }\| \Pi_{\alpha_j}^{\xi_j}\| \right) \\
& = &  \f{e^{-\xi^2(t-2 - \f{4}{e^t-1}) }}{(1-e^{-t})^n}. \nonumber
\eea
Since the both sides of (\ref{expansion}) are equal on the dense subspace spanned by  $\{\psi_\alpha^\xi; \alpha\in \bN\}$, an argument of density shows that they are equal on the whole space $L^2$.
\end{proof}

As a consequence of the proof of Proposition \ref{prop2.3}, we obtain the following estimate on the semigroup
\begin{cor} \label{cor2.4}
The following  estimate holds for $t>0$ and $\xi\in \bR^n$
\be \label{e2.14}
\|e^{-t \hat{P}_0(\xi)}\|_{\vB(L^2(\bR_v^n))}\leq\f{e^{-\xi^2(t-2 - \f{4}{e^t-1}) }}{(1-e^{-t})^n}
\ee
\end{cor}


\begin{thebibliography}{112}



\bibitem{aggmms}  A. Arnold, I. M. Gamba, M. P. Gualdani, S. Mischler, C. Mouhot, C. Sparber, The Wigner-Fokker-Planck equation: Stationary states and large-time behavior,  Math. Models Methods Appl. Sci. 22 (2012), no. 11, 1250034, 31 pp.

\bibitem{boul} L. S. Boulton, Non-self-adjoint harmonic oscillator, compact semigroups and pseudospectra. J. Operator Theory,  47(2002), no. 2, 413-429


\bibitem{DK04}E. B. Davies and A. B. J. Kuijlaars, Spectral Asymptotics of the Non-self-adjoint Harmonic Oscillator, J. London Math. Soc. (2)70(2004), 420-426.

\bibitem{d4} E. B. Davies, Linear Operators and Their Spectra, Web Supplement. Preprint, http://www.mth.kcl.ac.uk/staff/eb$_-$davies/LOTSwebsupp32W.pdf.

\bibitem{dsz} N. Dencker, J. Sj\"ostrand, M. Zworski, Pseudospectra of semiclassical (pseudo-)differential operators,
Commun. on Pure and Appl. Math., LVII(2004), 384-415.

\bibitem{dv} L. Desvillettes, C. Villani, On the trend to global equilibrium in spatially inhomogeneous entropy-dissipating systems: The linear Fokker-Planck equation, Commun. Pure Appl. Math., LIV(2001), 1-42.



\bibitem{hkn} B. Helffer, M.  Klein,  F.  Nier,  Quantitative analysis of metastability in reversible diffusion processes via a Witten complex approach. Mat. Contemp. 26 (2004), 41-85.

\bibitem{hln} B. Helffer,  F. Nier, Hypoelliptic estimates and spectral theory for Fokker-Planck
operators and Witten Laplacians. Lecture Notes in Mathematics, 1862. Springer-Verlag, Berlin, 2005. x+209 pp. ISBN: 3-540-24200-7

\bibitem{hrn} F. H\'erau, F. Nier, Isotropic hypoellipticity and trend to equilibrium for the
Fokker-Planck equation with a high-degree potential. Arch. Ration. Mech. Anal. 171 (2004), no. 2, 151-218.

\bibitem{hhs} F. Hérau,  M. Hitrik, J. Sjöstrand,  Tunnel effect for Kramers-Fokker-Planck type operators: return to equilibrium and applications. Int. Math. Res. Not. IMRN 2008, no. 15, Art. ID rnn057, 48 pp.

\bibitem{hss} F. Hérau, J. Sjöstrand,  C. Stolk, Semiclassical analysis for the Kramers-Fokker-Planck equation. Comm. Partial Differential Equations 30 (2005),  no. 4-6, 689-760. 

\bibitem{hor} L.  Hörmander,  The analysis of linear partial differential operators. III. Pseudo-differential operators.  Classics in Mathematics. Springer, Berlin, 2007.

\bibitem{Kat80}T. Kato, Perturbation Theory of Linear Operators, Springer, Berlin, 1980.

\bibitem{ls} A. Laptev, O. Safronov, Eigenvalues estimates for Schr\"odinger operators with complex potentials,  Comm. Math. Phys.  292  (2009),  no. 1, 29-54.

\bibitem{nak} S. Nakamura, Low energy asymptotics for Schrödinger operators with slowly decreasing potentials. Comm. Math. Phys. 161 (1994), no. 1, 63-76.

\bibitem{n} F. Nier,   Hypoellipticity for Fokker-Planck operators and Witten Laplacians. Lectures on the analysis of nonlinear partial differential equations.  pp. 31-84 in {\it  Morningside Lect. in Math.}, Vol. 1, ed. F. H. Lin, X. P. Wang and P. Zhang,  Int. Press, Somerville, MA, 2012, iv+317 pp. ISBN: 978-1-57146-235-0

\bibitem{risc} H. Risken,  The Fokker-Planck equation, Methods of solutions and applications.  Springer, Berlin, 1989.

\bibitem{v} C. Villani,   Hypocoercivity. Mem. Amer. Math. Soc. 202 (2009), no. 950, iv+141 pp. ISBN:
978-0-8218-4498-4

\bibitem{w0} X.P. Wang, Asymptotic expansion in time of the Schr\"odinger group on conical manifolds,  Ann.  Inst.  Fourier  56(6) (2006), 1903-1945.

\bibitem{w1} X. P. Wang, Number of eigenvalues for dissipative  Schr\"odinger operators under perturbation, 
 J. Math. Pures Appl. 96 (2011), 409-422

\bibitem{w2}X. P. Wang, Time-decay of semigroups generated by dissipative Schrödinger operators. J. Differential Equations 253 (2012), no. 12, 3523-3542.


\bibitem{yaf} D. Yafaev,  The low energy scattering for slowly decreasing potentials,  Commun.
Math Phys. 85 (1982), 117-196.


\end{thebibliography}
\end{document}